%% file: 0_Scattering.tex
\documentclass{amsart}

\makeatletter
\def\l@subsection{\@tocline{2}{0pt}{2.5pc}{5pc}{}}
\renewcommand\tocchapter[3]{%
  \indentlabel{\@ifnotempty{#2}{\ignorespaces#2.\quad}}#3%
}
\newcommand\@dotsep{4.5}
\def\@tocline#1#2#3#4#5#6#7{\relax
  \ifnum #1>\c@tocdepth 
  \else
    \par \addpenalty\@secpenalty\addvspace{#2}%
    \begingroup \hyphenpenalty\@M
    \@ifempty{#4}{%
      \@tempdima\csname r@tocindent\number#1\endcsname\relax
    }{%
      \@tempdima#4\relax
    }%
    \parindent\z@ \leftskip#3\relax \advance\leftskip\@tempdima\relax
    \rightskip\@pnumwidth plus1em \parfillskip-\@pnumwidth
    #5\leavevmode\hskip-\@tempdima{#6}\nobreak
    \leaders\hbox{$\m@th\mkern \@dotsep mu\hbox{.}\mkern \@dotsep mu$}\hfill
    \nobreak
    \hbox to\@pnumwidth{\@tocpagenum{#7}}\par
    \nobreak
    \endgroup
  \fi}
\makeatother
\AtBeginDocument{%
\makeatletter
\expandafter\renewcommand\csname r@tocindent0\endcsname{0pt}
\makeatother
}
\def\l@subsection{\@tocline{2}{0pt}{2.5pc}{5pc}{}}

\usepackage[english]{babel}
\usepackage[utf8x]{inputenc}
\usepackage[T1]{fontenc}
\usepackage[a4paper,top=3cm,bottom=4.5cm,left=3cm,right=3cm,marginparwidth=2cm]{geometry}
\usepackage{indentfirst}
\usepackage[dvipsnames]{xcolor}
\usepackage{enumitem}

\usepackage{amsmath}
\usepackage{amsthm}
\usepackage[color=pink]{todonotes}
\usepackage[colorlinks=true, linktocpage=true]{hyperref}
\usepackage{cleveref}
\usepackage{tikz-cd}
\usepackage{amsfonts}
\usepackage{amssymb}
\usepackage{graphicx}  
\usepackage{mathrsfs}

\theoremstyle{plain}
\newtheorem{theorem}{Theorem}[section]
\newtheorem{lemma}[theorem]{Lemma}
\newtheorem{corollary}[theorem]{Corollary}
\newtheorem{conjecture}[theorem]{Conjecture}

\newtheorem{proposition}[theorem]{Proposition}

\theoremstyle{definition}
\newtheorem{definition}[theorem]{Definition}
\newtheorem{remark}[theorem]{Remark}
\newtheorem{example}[theorem]{Example}
\numberwithin{equation}{section}

\author{Lang Mou}
\title{Scattering Diagrams of Quivers with Potentials and Mutations}
\date{}

\begin{document}
\maketitle

\begin{abstract}
    For a quiver with non-degenerate potential, we study the associated stability scattering diagram and how it changes under mutations. We show that under mutations the stability scattering diagram behaves like the cluster scattering diagram associated to the same quiver, which gives them identical cluster chamber structures. As an application, we prove if the quiver has a green-to-red sequence, these two scattering diagrams are identical and if the quiver comes from a once-punctured torus, they differ by a central wall-crossing. This verifies in those cases a conjecture of Kontsevich-Soibelman that a particular set of initial data given by a quiver determines the Donaldson-Thomas series of the quiver with a non-degenerate potential.
\end{abstract}

\tableofcontents

\section{Introduction}
\input{1_introduction.tex}

\section*{Acknowledgements}
The author thanks Eric Babson for his patience, encouragement and countless discussions. He would also like to thank Tom Bridgeland, Ben Davison, Bernhard Keller, Travis Mandel and Fan Qin for helpful discussions. He thanks Bernhard Keller for his comments and remarks.

\section{General theory of scattering diagrams}\label{definitions}
\input{2_diagram.tex}

\section{Review of cluster and stability scattering diagrams}\label{cluster}
\input{3_cluster.tex}

\section{Mutations}\label{muta}
\input{4_mutation.tex}

\section{Applications}\label{application}
\input{5_application.tex}

\bibliographystyle{abbrv}
\bibliography{scat}

\newcommand{\Addresses}
{{
  \bigskip
  \footnotesize
  Lang Mou, \textsc{Department of Mathematics, University of California Davis}\par\nopagebreak
  \textit{E-mail address}: \texttt{lmou@math.ucdavis.edu}
}}
\Addresses

\end{document}

%% file: 1_introduction.tex
Cluster algebras, introduced by Fomin and Zelevinsky in \cite{fomin2002cluster}, are a class of commutative algebras generated in some Laurent polynomial ring by a distinguished set of Laurent polynomials (the \textit{cluster variables}) grouped in overlapping subsets (the \textit{clusters}) recursively obtained by operations called \textit{mutations}. Important properties of cluster algebras regarding their bases are proved by Gross-Hacking-Keel-Kontsevich in \cite{gross2018canonical} where a technical tool named \textit{cluster scattering diagram} plays a crucial role. In general, a scattering diagram, sometimes under the name \textit{wall-crossing structure}, is a (possibly infinite) cone complex in a vector space where its codimension one cones (\textit{walls}) are decorated with certain transformations (elements in some group) referred as \textit{wall-crossings}. The cluster scattering diagram associated to a cluster algebra has maximal cells (the \textit{cluster chambers}) corresponding to clusters and the walls of the cluster chambers are decorated with wall-crossings to describe mutations. A nicely behaved basis of a cluster algebra (the \textit{canonical basis}) is constructed in \cite{gross2018canonical} by counting in the associated cluster scattering diagram certain piecewise linear curves (the \textit{broken lines}) which bend when crossing walls. The canonical basis contains monomials of cluster variables in the same cluster (the \textit{cluster monomials}).

A large class of cluster algebras of interest are those associated to quivers. To cluster algebras of this type, there is a seemingly different approach which utilizes a categorification modeled on quiver representations. For more details in this approach (the \textit{additive categorification}), we refer to the nice survey \cite{keller2008cluster}. In view of the additive categorificaiton, clusters correspond to t-structures of the relevant triangulated category and mutations are essentially tiltings of the t-structures. In this framework, cluster monomials, part of the canonical basis, can be recovered by applying Caldero-Chapoton type formulas to certain quiver representations; see \cite{caldero2006cluster} for finite type quivers, generalizations in \cite{caldero2008triangulated, caldero2006triangulated, palu2008cluster},  \cite{derksen2010quivers, plamondon2011cluster} for arbitrary 2-acyclic quivers and \cite{nagao2013donaldson} for a point of view closest to this paper.

It is interesting to ask what are the meanings of cluster scattering diagrams, wall-crossings and broken lines in the framework of additive categorification? An important first step towards an answer is taken by Bridgeland, who constructs in \cite{bridgeland2016scattering} a \textit{stability scattering diagram} (\cref{qstabilityscatteringdiagram}) for each quiver with potential using representations of the corresponding Jacobian algebra. He shows that for an acyclic quiver (thus with only zero potential), the stability scattering diagram is identical to the corresponding cluster scattering diagram. In this case, the Caldero-Chapoton formulas for cluster monomials thus have interpretations in both scattering diagrams. However, even in this case, the representation theoretic meaning of the canonical basis (apart from cluster monomials) is still unclear.

One goal of this paper is to further investigate the relationship between these two related but a priori not necessarily identical scattering diagrams, setting foundations towards a better understanding of the categorical meanings of the combinatorial objects extracted from cluster scattering diagrams, especially the categorification of broken lines.

We closely follow the setting in \cite{gross2018canonical} but restrict ourselves to the cases associated to quivers; that is, we fix a lattice $N\cong \mathbb Z^n$ equipped with a $\mathbb Z$-valued skewsymmetric (instead of more generally skewsymmetrizable) form. A seed $\mathbf{s}$ is a basis of $N$ and there is a quiver $Q(\mathbf{s})$ whose adjacency matrix is the matrix of pairings of $\mathbf{s}$ in the skewsymmetric form. To the quiver $Q(\mathbf{s})$, there is an associated cluster scattering diagram $\mathfrak D_\mathbf{s}$ (\cref{clustersd}). A seed with potential $(\mathbf{s,w})$ is a seed $\mathbf{s}$ together with a potential $\mathbf{w}$ of the quiver $Q(\mathbf{s})$. To the quiver with potential $(Q(\mathbf{s}),\mathbf{w})$, we have the associated stability scattering diagram $\mathfrak D_\mathbf{s,w}$ (\cref{qstabilityscatteringdiagram}). The mutation $\mu_k^{-}$ (defined in \cref{mutationofsp}, corresponding to $\mu_k$ in \cite[section 1.3]{gross2018canonical}) transforms $(\mathbf{s,w})$ into another seed with potential $\mu_k^-(\mathbf{s,w})$ where $k$ is one of the vertices of the quiver $Q(\mathbf {s})$. Our first main result is the following mutation formula of stability scattering diagrams. This theorem should be compared with \cite[theorem 1.24]{gross2018canonical}. For the precise statements, see \cref{scatmutation}, \cref{chamberstructure} and \cref{agreecc}.
\begin{theorem}\label{theorem1}
Let $(\mathbf{s,w})$ be a non-degenerate seed with potential. 

\begin{enumerate}[label=(\roman*)]
    \item we have the following equality between stability scattering diagrams: 
\[
\mathfrak D_{\mu_k^-(\mathbf{s,w})} = T_k(\mathfrak D_{\mathbf{s,w}})
\]
where $T_k$ is certain piecewise linear transformation on the scattering diagram $\mathfrak D_{\mathbf{s,w}}$.
    \item The scattering diagrams $\mathfrak D_\mathbf{s,w}$ and $\mathfrak D_\mathbf{s}$ have the same cluster chambers and the same wall-crossings on walls of cluster chambers. 
\end{enumerate}
\end{theorem}

We have the following corollary in comparing stability scattering diagrams with cluster scattering diagrams. Note that Qin also has a proof of (i) of \cref{theorem2} using \textit{opposite} scattering diagrams \cite{qin2019bases}.

\begin{corollary}\label{theorem2}
Let $(\mathbf{s,w})$ be a non-degenerate seed with potential.
\begin{enumerate}[label=(\roman*)]

\item  If the quiver $Q(\mathbf{s})$ has a green-to-red sequence, then the stability scattering diagram $\mathfrak D_\mathbf{s,w}$ is equal to the cluster scattering diagram $\mathfrak D_{\mathbf{s}}$. See \cref{reddening}.

\item If the quiver $Q(\mathbf{s})$ comes from a once-punctured torus (the Markov quiver), then $\mathfrak D_\mathbf{s,w}$ and $\mathfrak D_\mathbf{s}$ differ by a central wall-crossing on the separating hyperplane. See \cref{markovquiver}.
\end{enumerate}
\end{corollary}

The above theorems have analogues for scattering diagrams that are relevant to the Donaldson-Thomas theory of quivers with potentials \cite{kontsevich2008stability, joyce2012theory}. For a seed $\mathbf{s}$, one can associated a scattering diagram ${\mathfrak{D}}^\mathrm{In}_\mathbf{s}$ (\cref{dtscattering}) using a particular set of initial data similar to the one in defining the cluster scattering diagram $\mathfrak D_\mathbf{s}$. On the other hand, the counterpart of $\mathfrak D_\mathbf{s,w}$ is the \textit{Donaldson-Thomas scattering diagram} ${\mathfrak D}^\mathrm{DT}_{\mathbf{s,w}}$ (\cref{dtscattering}) which encodes the Donaldson-Thomas invariants for the quiver with potential $(Q(\mathbf{s}),\mathbf{w})$ (while $\mathfrak D_\mathbf{s,w}$ encodes the Joyce invariants). The following conjecture is proposed by Kontsevich-Soibelman in \cite[Conjecture 3.3.4]{kontsevich2014wall}. For the precise statement, see conjecture \ref{ksconjecture}.

\begin{conjecture}
For a seed with non-degenerate potential $(\mathbf{s,w})$, the DT scattering diagram ${\mathfrak D}^\mathrm{DT}_{\mathbf{s,w}}$ and the scattering diagram $\mathfrak D^\mathrm{In}_\mathbf{s}$ differ by a central element in the group of wall-crossings. 
\end{conjecture}

As an application of the proofs of \cref{theorem2}, we confirm this conjecture for quivers with green-to-red sequences and also for the Markov quiver (\cref{dtreddening} and \cref{markovquiver}). Of course, one can formulate the same conjecture for $\mathfrak D_\mathbf{s,w}$ and $\mathfrak D_\mathbf{s}$. Moreover, these two conjectures are both mutation-invariant; see \cref{mutationinvariance} for the following proposition.

\begin{proposition}
The above conjecture is true for the scattering diagrams $\mathfrak D_{\mathbf{s,w}}$ and $\mathfrak D_\mathbf{s}$ (resp. ${\mathfrak D}^\mathrm{DT}_\mathbf{s,w}$ and ${\mathfrak D}^\mathrm{In}_\mathbf{s}$) if and only if it is true for their mutations $\mathfrak D_{\mu^-_k(\mathbf{s,w})}$ and $\mathfrak D_{\mu^-_k(\mathbf{s})}$ (resp. ${\mathfrak D}^\mathrm{DT}_{\mu_k^-(\mathbf{s,w})}$ and ${\mathfrak D}^\mathrm{In}_{\mu_k^-(\mathbf{s})}$). \end{proposition}

We conclude with the contents of this paper. General definitions and properties of scattering diagrams are given in \cref{definitions}. Then we review cluster and stability scattering diagrams and their quantizations in \cref{cluster}. Almost every of the aforementioned results can be extended to the quantum case. Our main technical tool is developed and the main results are proved in \cref{muta}. The applications in cluster algebras and Donaldson-Thomas theory are explained in \cref{application}.

%% file: 2_diagram.tex
In this section, we introduce the definition of consistent $\mathfrak g$-valued scattering diagrams (consistent $\mathfrak g$-SD in short) for a graded Lie algebra $\mathfrak g$ following Kontsevich-Soibelman \cite{kontsevich2014wall}. To be precise, the case we consider in this paper is corresponding to in loc. cit. the \textit{wall-crossing structures on a vector space} with the assumption that the support of the Lie algebra $\mathfrak g$ is contained in a strict convex cone. The main result in this section is the existence of a canonical cone complex structure of an arbitrary consistent $\mathfrak g$-SD with finite support; see \cref{minimalsupport}. This feature of consistent scattering diagrams is important in applications in cluster theory and we hope it to be also useful for studying connections to toric geometry as in \cite{gross2010tropical}.

\subsection{Graded Lie algebras} \label{setting}
\subsubsection{} Let $N\cong \mathbb Z^n$ be an $n$-dimensional lattice, i.e. a free abelian group of rank $n\in \mathbb N$. Fix a basis $\mathbf{s} = \{s_1,\dots, s_n\}$ of $N$. We define $N^+ = N^+_\mathbf{s}$ to be the subsemigroup (without $0$) of $N$ non-negatively generated by $\mathbf{s}$. The monoid $N_\mathbf{s}^\oplus$ is $N_\mathbf{s}^+\cup\{0\}$. We will denote an $N^+$-graded Lie algebra by $\mathfrak g$. That is,
\[
\mathfrak g=\bigoplus_{d\in N^+}\mathfrak g_d
\]
as a free module over a commutative algebra over $\mathbb Q$ (usually a vector space over a field $k$ of characteristic zero) with a Lie bracket such that $[\mathfrak g_{n_1},\mathfrak g_{n_2}]\subset \mathfrak g_{n_1+n_2}$ for any $n_1, n_2\in N^+$. For a subset $S$ of $N^+$, we will usually denote by $\mathfrak g_S$ the direct sum of the homogeneous spaces supported on $S$, i.e. 
\[
\mathfrak g_S\colon = \bigoplus_{d\in S}\mathfrak g_S.
\]

Every ideal $I$ of the semigroup $N^+$ gives an ideal $\mathfrak g_I$ of the Lie algebra $\mathfrak g$ and a quotient Lie algebra
\[
\mathfrak g^{<I}\colon =\mathfrak g/\mathfrak g_I.
\]
Note that the Lie algebra $\mathfrak g^{<I}$ is still $N^+$-graded and is supported on the set $N^+\setminus I$. If we have an inclusion of ideals $I\subset J$ of $N^+$, then there is an induced $N^+$-graded Lie algebra homomorphism $\rho_{I,J}\colon \mathfrak g^{<I}\rightarrow \mathfrak g^{<J}$.

When the Lie algebra $\mathfrak g$ is nilpotent, there is a corresponding unipotent algebraic group $G$ which is in bijection with $\mathfrak g$ as sets. The product in the group $G$ is given by the Baker–Campbell–Hausdorff formula. The bijection is denoted by $\mathrm {exp}\colon \mathfrak g\rightarrow G$. If the Lie algebra $\mathfrak g$ has finite support, i.e. when \[\mathrm{Supp}(\mathfrak g)\colon = \{d\in N^+|\ \mathfrak g_d\neq 0\}\] is a finite set, then it is nilpotent. We say an ideal $I$ of $N^+$ is \textit{cofinite} if $N^+\setminus I$ is a finite set and denote the set of all cofinite ideals by $\mathrm{Cofin}(N^+)$. In this case, the quotient Lie algebra $\mathfrak g^{<I}$ has finite support and thus is nilpotent, giving the corresponding unipotent group $G^{<I}$. The inclusion of cofinite ideals $I\subset J$ induces a quotient map between groups which we also denote by $\rho_{I,J}\colon G^{<I}\rightarrow G^{<J}$. In fact, we can define an order $J\leq I$ for $I\subset J$. Then the set of cofinite ideals $\mathrm{Cofin}(N^+)$ becomes a directed set and the associations $I\mapsto \mathfrak g^{<I}$ and $I\mapsto  G^{<I}$ become inverse systems indexed by $\mathrm{Cofin}(N^+)$. Taking the projective limits, we obtain a pro-nilpotent Lie algebra and a corresponding pro-unipotent algebraic group:
\begin{equation}\label{socfi}
\hat {\mathfrak g}\colon =\lim_{\underset{I}{\longleftarrow}}\mathfrak g^{<I}\cong\prod_{n\in N^+}\mathfrak g_n\quad \mathrm{and}\quad \hat{G}\colon =\lim_{\underset{I}{\longleftarrow}} G^{<I}.
\end{equation}
The group $\hat G$ is again in bijection with the Lie algebra $\hat {\mathfrak g}$ as sets.

We put
\[
M\colon =\mathrm{Hom} (N,\mathbb Z)\cong Z^n\quad \mathrm{and}\quad  M_\mathbb R\colon =M\otimes \mathbb R\cong \mathbb R^n.
\]
For any $m\in M_\mathbb R$, there is a partition of $N^+\colon$ $$N^+ = P_{m,+}\sqcup P_{m,0}\sqcup P_{m,+}$$ where
\begin{equation}\label{supportdecomposition}
P_{m,\pm}\colon =\{n\in N^+|\ m(n)\gtrless 0\}\ 
 \ \mathrm{and}\ P_{m,0}\colon=\{n\in N^+|\ m(n)=0\}.    
\end{equation}
This partition of $N^+$ induces a decomposition of $\mathfrak g\colon$
\begin{equation}\label{decompositionliealgebra}
\mathfrak g = \mathfrak g_{m,+}\oplus \mathfrak g_{m,0} \oplus \mathfrak g_{m,-}    
\end{equation}
where $\mathfrak g_{m,\bullet} \colon = \mathfrak g_{P_{m,\bullet}}$ is a graded Lie subalgebra for $\bullet\in \{0,+,-\}$. We denote the corresponding pro-unipotent subgroups by $\hat G_{m,\bullet}$. In the following lemma, we fix some $m\in M_\mathbb R$ so $m$ is omitted in the subscript.

\begin{lemma}\label{factorization}
  Fix some $m\in M_\mathbb R$. Then the decomposition (\ref{decompositionliealgebra}) induces a unique factorization of any element $g\in \hat G$ into $g = g_{+}\cdot g_0\cdot g_-$ where $g_\bullet\in \hat G_{\bullet}$ for $\bullet \in \{0,+,-\}$. In other words, the map $\phi\colon \hat G_{+}\times \hat G_{0}\times \hat G_{-}\rightarrow \hat G$ defined by
\[
\phi(g_+,g_0,g_-) = g_+\cdot g_0 \cdot g_-
\]
is a set bijection.
\end{lemma}

\begin{proof}
We first prove this for any $I\in \mathrm{Cofin}(N^+)$.
Take a filtration of cofinite ideals
\[
I = I_k \subset I_{k-1}\subset \cdots \subset I_0 = N^+
\]
such that $I_{i}\setminus I_{i+1} = \{n_i\}$ contains only one element. We have the following commutative diagram
\[\begin{tikzcd}
G_{+}^{<I_{i+1}}\times G_{0}^{<I_{i+1}}\times G_{-}^{<I_{i+1}}\arrow[r, "\phi_{i+1}"]\arrow[d, "f_i"]&G^{<I_{i+1}}\arrow[d, "h_i"]\\
G_{+}^{<I_{i}}\times G_{0}^{<I_{i}}\times G_{-}^{<I_{i}}\arrow[r, "\phi_{i}"]&G^{<I_{i}}
\end{tikzcd}
\]
where $f_{i}$ and $h_i$ are given by natural quotient maps induced by the inclusion $I_{i+1}\subset I_i$. Note that the maps $f_i$ and $g_i$ are both fibrations of $\exp({\mathfrak g_{n_i}})$-torsors and the map $\phi_{i+1}$ is $\exp(\mathfrak g_{n_i})$-equivariant. Here $\mathfrak g_{n_i}$ denotes the quotient Lie algebra $\mathfrak g^{<I_{i+1}}/\mathfrak g^{<I_{i}}$. Therefore, the bijection of $\phi_i$ would imply the bijection of $\phi_{i+1}$. Since $\phi_0$ is a bijection, the map 
\[
	\phi_I = \phi_k\colon G_{+}^{<I}\times G_{0}^{<I}\times G_{-}^{<I}\rightarrow G^{<I}
\]
is a bijection by induction.

Now we have similar commutative diagrams for any inclusion $I\subset J$
\[\begin{tikzcd}
G_{+}^{<I}\times G_{0}^{<I}\times G_{-}^{<I}\arrow[r, "\phi_I"]\arrow[d, "f_{I,J}"]&G^{<I}\arrow[d, "\rho_{I,J}"]\\
G_{+}^{<J}\times G_{0}^{<J}\times G_{-}^{<J}\arrow[r, "\phi_{J}"]&G^{<J}.
\end{tikzcd}
\]
Then the bijection extends to the projective limits.
\end{proof}
The factorization in the above lemma defines projection maps (of sets)
\begin{equation}\label{projectionmaps}
\pi_{m,\bullet}\colon \hat G\rightarrow \hat G_{m,\bullet}	
\end{equation}
by sending $g$ to $g_{m,\bullet}$. We will simply write $\pi_m$ for $\pi_{m,0}$.

\subsubsection{} By a \textit{cone} in the vector space $M_\mathbb R$, we mean a subset closed under scaling by $\mathbb R_{>0}$. A cone is \textit{convex} if it is convex as a subset of $M_\mathbb R$. A \textit{polyhedral cone} is a closed convex subset of $M_\mathbb R$ of the form
\[
\sigma = \left\{\sum_{i=1}^k\lambda_iv_i\mid \lambda_i\in \mathbb R_{\geq 0},\, v_i\in M_\mathbb R\right\}.
\]
It is called a \textit{rational polyhedral cone} if $v_i$ is in $M$ for each $i$. A face of a cone $\sigma$ is a subset of the form
\[
\sigma\cap n^\perp = \{m\in \sigma\mid m(n) = 0\}
\]
where $n\in N_\mathbb R$ satisfies $m(n)\geq 0$ for all $m\in \sigma$. A face of a cone is again a cone.

\begin{definition}
A \textit{cone complex} $\mathfrak S$ in $M_\mathbb R$ is a collection of rational polyhedral cones in $M_\mathbb R$ such that
\begin{enumerate}[label = (\roman*)]
\item for any $\sigma\in \mathfrak S$, if $\tau \subset \sigma$ is a face of $\sigma$, then $\tau\in \mathfrak S$;
\item for any $\sigma_1, \sigma_2\in \mathfrak S$, $\sigma_1\cap \sigma_2$ is a face of $\sigma_1$ and $\sigma_2$. 
\end{enumerate}
\end{definition}

Note that we do not require the cones in a cone complex to be \textit{strictly convex}. For example, a closed half space is allowed. We also do not require the collection to be finite, for which we call $\textit{a finite cone complex}$. If the union of all cones $|\mathfrak S|$ equals $M_\mathbb R$, we say that the cone complex $\mathfrak S$ is \textit{complete}. The complement of the union of all proper faces of $\sigma$ in $\sigma$ is called the \textit{relative interior} of $\sigma$ and is denoted by $\sigma^\circ$. It is relatively open, i.e. open in the subspace in $M_\mathbb R$ spanned $\sigma$. The set of cones $\sigma^\circ$ for all $\sigma\in \mathfrak S$ is denoted by $\mathfrak S^\circ$.

A cone complex $\mathfrak S$ is also a poset with the partial order $\sigma_1\prec \sigma_2$ if and only if $\sigma_1$ is a face of $\sigma_2$. It can also be viewed as a category where the only morphisms are of the form $\sigma_1 \prec \sigma_2$. We define
\begin{equation}\label{star}
	\mathrm {Star} (\sigma) \colon = \{\tau^\circ \in \mathfrak S^\circ\mid \sigma\prec \tau\}
\end{equation}
and
\[
|\mathrm{Star}(\sigma)|\colon = \bigcup_{\tau^\circ \in \mathrm {Star}(\sigma)}\tau^\circ.
\]


\begin{example}[Hyperplane arrangements]\label{hyperplanearrangements}
	Let $S$ be a finite subset of $N$. Consider a partition $P$ of $S$ into three disjoint subsets
	\[
	S = P_+ \sqcup P_0 \sqcup P_-. 
	\]
	We define
	\[
\sigma_P\colon = \overline{\{m\in M_\mathbb R|\ m(P_+)> 0,\ m(P_0) = 0\ \mathrm{and}\ m(P_-)< 0\}}.
	\]
	One easily checks the (non-empty) cones $\sigma_P$ for all such partitions of $S$ form a complete cone complex $\mathfrak S_S$ in $M_\mathbb R$. We have $\sigma_{P_1}\prec \sigma_{P_2}$ if and only if
	\[
	P_1\prec P_2,\ \mathrm{i.e.}\ P_{1,+}\subset P_{2,+},\ P_{2,0}\subset P_{1,0},\ \mathrm{and}\ P_{1,+}\subset P_{2,+}.
	\]
\end{example}

\subsubsection{}\label{sectionofkeylemma} Assume that $S = \mathrm{Supp}(\mathfrak g)$ is finite. Let $\sigma \in \mathfrak S_S$ and set $$\sigma ^\perp = \{d\in N\mid m(d) = 0\quad \forall m \in \sigma \}.$$ We put $\mathfrak g_\sigma\colon  = \mathfrak g_{\sigma^\perp \cap S}$. If $\sigma_1\prec \sigma_2$, we define a map $\pi_{\sigma_1,\sigma_2}\colon \exp(\mathfrak g_{\sigma_1})\rightarrow \exp(\mathfrak g_{\sigma_2})$ as follows. Suppose $\sigma_i = \sigma_{P_i}$ for $i=1, 2$. Then we have $P_1\prec P_2$. Let $m$ be in $\sigma_2^\circ$ and it gives (independent of the choice of $m$) a partition
\[
P_{1,0} = (P_{1,0})_{m,+}\sqcup P_{2,0} \sqcup (P_{1,0})_{m,-}
\]
as in (\ref{supportdecomposition}). Note that $\mathfrak g_{\sigma_i} = \mathfrak g_{P_{i,0}}$. Then one has the projection $\pi_{\sigma_1,\sigma_2}\colon \exp(\mathfrak g_{\sigma_1})\rightarrow \exp(\mathfrak g_{\sigma_2})$ exactly as in \cref{factorization}. For example, if $\sigma_1$ be the smallest cone in $\mathfrak S_S$ with respect to the order $\prec$, $\pi_{\sigma_1,\sigma_2} = \pi_m$ for any $m\in \sigma_2^\circ$. The following key lemma will be used later.
\begin{lemma}\label{keyfunctor}
	Let $\mathfrak g$ be an $N^+$-graded Lie algebra with finite support $S$. The assignment $\sigma \mapsto \exp(\mathfrak g_{\sigma})$, $(\sigma_1 \prec \sigma_2)\mapsto \pi_{\sigma_1,\sigma_2}$ defines a functor from $\mathfrak S_S$ to $\mathrm {Grp}$ the category of groups.
\end{lemma}
\begin{proof}
	The only thing we need to check is if $\sigma_0\prec \sigma_1\prec \sigma_2$, then
	\[
	\pi_{\sigma_1,\sigma_2}\circ \pi_{\sigma_0,\sigma_1} = \pi_{\sigma_0,\sigma_2}.
	\]
	Without loss of generality, we assume that $\sigma_0$ is the origin. Then $\mathfrak g_{\sigma_0} = \mathfrak g$ and for $m_i\in \sigma_i^\circ$, $\pi_{\sigma_0,\sigma_i} = \pi_{m_i}$ for $i= 1,2$ (see (\ref{projectionmaps})). Suppose $\sigma_i$ comes from a partition $P_i$ as before.
	We know $\pi_{\sigma_0,\sigma_1}(g)$ for $g\in G = \exp(\mathfrak g)$ is the middle term in the factorization (\cref{factorization})
	\[
	g = \pi_{m_1,+}(g)\cdot\pi_{m_1}(g)\cdot \pi_{m_1,-}(g).
	\]
	By factorizing $g_1 = \pi_{m_1}(g)$ further with respect to $m_2$, we get $\pi_{\sigma_1,\sigma_2}\circ \pi_{\sigma_0,\sigma_1}(g) = \pi_{m_2}(g_1)$:
	\[
	g_1 = \pi_{m_2,+}(g_1)\cdot \pi_{m_2}(g_1)\cdot \pi_{m_2,-}(g_1).
	\]
	Using previous notations, we have
	\[
	 P_{1,0} = (P_{1,0})_{m_2,+}\sqcup P_{2,0} \sqcup (P_{1,0})_{m_2,-}
	\]
	and
	\[
	P_{2,+} = P_{1,+}\sqcup (P_{1,0})_{m_2,+},\quad P_{2,-} = P_{1,-}\sqcup (P_{1,0})_{m_2,-}.
	\]
	This implies 
	\[
	\pi_{m_1,+}(g)\cdot \pi_{m_2,+}(g_1) \in \exp(\mathfrak g_{P_{2,+}}) = G_{m_2,+}
	\]
	and 
	\[
	\pi_{m_2,-}(g_1)\cdot \pi_{m_1,-}(g)\in \exp(\mathfrak g_{P_{2,-}}) = G_{m_2,-}.
	\]
	Therefore we obtain a factorization
	\[
	g = (\pi_{m_1,+}(g)\cdot \pi_{m_2,+}(g_1))\cdot \pi_{\sigma_1,\sigma_2}\circ \pi_{\sigma_0,\sigma_1}(g) \cdot (\pi_{m_2,-}(g_1)\cdot \pi_{m_1,-}(g)).
	\]
	Such a factorization is unique by \cref{factorization}. We conclude that the middle term $\pi_{\sigma_1,\sigma_2}\circ \pi_{\sigma_0,\sigma_1}(g)$ equals $\pi_{m_2}(g) = \pi_{\sigma_0,\sigma_2}(g)$. 
\end{proof}

\subsection{Wall-crossing structures}

In this section, we define \textit{wall-crossing structures} following \cite{kontsevich2014wall}. Readers can safely skip this section to \cref{consistent}. 
\subsubsection {} We first assume that $S = \mathrm{Supp}(\mathfrak g)$ is finite. Define the following set (the étale space)
\[
\mathcal G^{\text{ét}}\colon = \{(m,g')\mid m\in M_\mathbb R,\, g'\in G_{m,0}\}
\]
with a collection of subsets 
\[
W_{g,U}\colon = \{(m,g')\mid m\in U,\, g' = \pi_m(g) \}
\]
where $g$ runs through $G$ and $U$ runs through all open sets of $M_\mathbb R$.

\begin{lemma}
The subsets $W_{g,U}$ give a base of topology on $\mathcal G^{\text{ét}}$. The projection $\mathcal G^{\text{ét}}\rightarrow M_\mathbb R$ defined by $(m,g')\mapsto m$ is a local homeomorphism.
\end{lemma}

\begin{proof}
First of all, the subsets $W_{g,U}$ cover $\mathcal G^{\text{ét}}$. Consider the cone complex $\mathfrak S_S$ defined in \cref{hyperplanearrangements}. Let $W_{g, U}$ and $W_{h,V}$ be two subsets of $\mathcal G^\text{ét}$. By \cref{keyfunctor}, if $\pi_m(g) = \pi_m(h)$ for some $m$, then $\pi_{m'}(g) = \pi_{m'}(h)$ for any $m'$ in the interior of $|\mathrm{Star}(\sigma)|$ where $\sigma$ is the only cone in $\mathfrak S_S$ such that $m\in \sigma^\circ$. It follows that the set $\{m\in M_\mathbb R\mid \pi_m(g) = \pi_m(h)\}$ is open in $M_\mathbb R$ and correspondingly $W_{g, U}\cap W_{h,V}$ is a set of the same form. Thus the collection of subsets $W_{g,U}$ form a base of topology.

For the second statement, at $(m,g')\in \mathcal G^{\text{ét}}$, we take the subset $W_{g',U}$ where $U$ is open and contains $m$. It is clear that the map $W_{g',U}\rightarrow U$ is a homeomorphism.
\end{proof}

\begin{definition}\label{definition}
The \textit{sheaf of wall-crossing structures} $\mathcal{WCS}_{\mathfrak g}$ of the Lie algebra $\mathfrak g$ is defined to be the sheaf of sections of the local homeomorphism $\mathcal G^{\text{ét}}\rightarrow M_\mathbb R$ in the above lemma.
\end{definition}

For any $g\in G$, there is a section $$s_g\colon M_\mathbb R\rightarrow \mathcal G^{\text{ét}},\quad m\mapsto (m, \pi_m(g)).$$ The image of $s_g$ is $W_{g,M_\mathbb R}$ and $s_g$ is clearly a global section of the sheaf $\mathcal {WCS}_\mathfrak g$. By the construction of the étale space $\mathcal G^{\text{ét}}$, the stalk $(\mathcal{WCS}_\mathfrak g)_m$ is canonically identified with $G_{m,0}$. The germ of the section $s_g$ at $m$ is just given by $\pi_m(g)\in G_{m,0}$.

Consider the map $s\colon G\rightarrow \Gamma(\mathcal {WCS}_\mathfrak g,M_\mathbb R)$ sending $g$ to the global section $s_g$.  The following lemma follows from corollary 2.1.3 and lemma 2.1.7 in \cite{kontsevich2014wall}.

\begin{lemma}\label{bijection}
The map $s\colon G\rightarrow \Gamma(\mathcal {WCS}_\mathfrak g,M_\mathbb R)$, $g\mapsto s_g$ is a bijection. Thus the set of global sections of $\mathcal{WCS}_\mathfrak g$ is in bijection with the group $G$.
\end{lemma}

\subsubsection{} Now we remove the finiteness restriction of the set $\mathrm{Supp}(\mathfrak g)$. There is a directed inverse system of sheaves $\mathcal{WCS}_{\mathfrak g^{<I}}$ indexed by the directed set $\mathrm{Cofin}(N^+)$
\begin{equation}\label{inversesystem}
    \mathrm{Cofin}(N^+)\rightarrow Sh(M_\mathbb R),\quad I\mapsto \mathcal {WCS}_{\mathfrak g^{<I}}
\end{equation}
induced by natural quotient maps of Lie algebras. 

\begin{definition}\label{sheafofwallcrossing}
In general, the \textit{sheaf of wall-crossing structures} of the Lie algebra $\mathfrak g$ is defined to be the projective limit of the above mentioned inverse system (\ref{inversesystem}) in the category of sheaves of sets on $M_\mathbb R$:
\[
\mathcal{WCS}_\mathfrak{g}\colon = \lim_{\underset{I}{\longleftarrow}}\mathcal{WCS}_{\mathfrak g^{<I}}\in Sh(M_\mathbb R).
\]
\end{definition} 

It is standard that the space of sections of the projective limit of sheaves is in bijection with the projective limit of spaces of sections. So we have the identification of the space of global sections with the pro-unipotent group $\hat G$:
\begin{equation}\label{globalsection}
\Gamma (\mathcal {WCS}_\mathfrak g, M_\mathbb R) = \lim_{\underset{I}{\longleftarrow}}\Gamma (\mathcal{WCS}_{\mathfrak g^{<I}},M_\mathbb R) = \lim_{\underset{I}{\longleftarrow}} G^{<I} = \hat G.    
\end{equation}
Given a global section $g\in \hat G$, we define the following map:
\begin{equation}\label{thefunction}
\phi_g\colon M_\mathbb R\rightarrow \hat G, m\mapsto \pi_m(g)\in \hat G_{m,0}\subset \hat G  
\end{equation}
which records the germs of the global section $g$ at every stalk. 

\begin{definition}[$\mathfrak g$-WCS cf. \cite{kontsevich2014wall}]\label{defscat}
A global section of the sheaf $\mathcal {WCS}_\mathfrak g$ is called a $\mathfrak g$-\textit{wall-crossing structure} ($\mathfrak g$-WCS in short). Thus the equality (\ref{globalsection}) shows $\Gamma(\mathcal{WCS}_\mathfrak g, M_\mathbb R)$, the set of all $\mathfrak g$-WCS', is in bijection with the pro-unipotent group $\hat G$. 
\end{definition}

\subsection{Consistent scattering diagrams}\label{consistent}
In this section we give the definition of a \textit{consistent $\mathfrak g$-scattering diagram} (\textit{consistent} $\mathfrak g$\textit{-SD} in short). The relation with $\mathfrak g$-WCS will be explained in \cref{equivalenceoftwodefinitions}.

\subsubsection{}As usual we first assume $S = \mathrm {Supp} (\mathfrak g)$ to be finite. 
\begin{definition}[$\mathfrak g$-SD]\label{equivalentdefinition}
A $\mathfrak g$-\textit{scattering diagram} ($\mathfrak g$-SD in short) $\mathfrak D$ with $S = \mathrm {Supp} (\mathfrak g)$ is a pair $(\mathfrak S_S, \phi_\mathfrak D)$ consisting of the cone complex $\mathfrak S_S$ (\cref{hyperplanearrangements}) and a function $\phi_{\mathfrak D}\colon  \mathfrak S_S\rightarrow  G$ for any $\sigma\in \mathfrak S_S$, $\phi_\mathfrak D(\sigma)$ is in the subgroup $\exp(\mathfrak g_{\sigma})\subset G$. If $\mathrm{codim}\ \sigma = 1$ (resp. $>1$), we call $\phi_\mathfrak D (\sigma)$ the \textit{wall-crossing (resp. face-crossing)} at $\sigma$.  
\end{definition}

For each cone $\sigma = \sigma_P$ in $\mathfrak S_S$ of codimension at least one, there is a maximal cell $\sigma^+$ relative to $\sigma$. It is the relative interior of the cone $\sigma_{P'}$ given by the partition $P'$ such that
\[
P'_{+} = P_{+}\sqcup P_0,\ P'_0 = \emptyset,\ \mathrm{and}\ P'_- = P_{-}.
\]
Similarly there is the negative maximal cell $\sigma^-$. For example, if $\sigma$ is a wall (i.e. $\mathrm{codim}\ \sigma = 1$), then $\sigma^+$ and $\sigma^-$ are the two maximal cells on the two sides of $\sigma$.

\begin{definition}\label{path}
A $\mathfrak S_S$-\textit{path} is a smooth curve $\gamma\colon [0,1]\rightarrow M_\mathbb R$ such that if for some $t\in [0,1]$, $\gamma(t)\in \sigma^\circ$ for some cone $\sigma$ of codimension at least one in $\mathfrak S_S$, then there exists a neighborhood $(t-\epsilon,t+\epsilon)$ of $t$ such that $\gamma((t-\epsilon, t))\subset \sigma^-$ and $\gamma((t,t+\epsilon))\subset \sigma^+$. We call $t$ a $\textit{negative crossing}$ and denoted by $t^-$ if $\epsilon >0$  and a \textit{positive crossing} $t^+$ if $\epsilon <0$ respectively.
\end{definition}

Let $\mathfrak D$ be a $\mathfrak g$-SD and $\gamma$ be a $\mathfrak S_S$-path $\gamma$ with finitely many crossings 
\[
0<t_1^{\epsilon_1}<t_2^{\epsilon_2}<\cdots<t_k^{\epsilon_k}<1
\]
where $\epsilon_i\in\{-,+\}$. Record the cones at these crossings by $\sigma_i$ the corresponding wall-crossings or face-crossings by
\[
g_i = 
\begin{cases}
 	\phi_\mathfrak D(\sigma_i)\quad &\mathrm{if}\quad \epsilon_i = +\\
 	\phi_\mathfrak D(\sigma_i)^{-1}\quad &\mathrm{if}\quad \epsilon_i = -\\
\end{cases}.
\]

\begin{definition}\label{pathorderedproduct}
The \textit{path-ordered product} for a $\mathfrak S_S$-path $\gamma$ in $\mathfrak D$ is defined as
\[
\mathfrak p_\gamma (\mathfrak D)\colon = g_k\cdot \cdots \cdot g_2\cdot g_1 \in G.
\]
\end{definition}

\begin{definition}[Consistent $\mathfrak g$-SD]\label{consistentscatteringdiagram}
A  $\mathfrak g$-SD $\mathfrak D$ is said to be a \textit{consistent} $\mathfrak g$-SD if the path-ordered product $\mathfrak p_\gamma(\mathfrak D)$ for any $\mathfrak S_S$-path $\gamma$ only depends on the end points $\gamma (0)$ and $\gamma (1)$.	
\end{definition}

There is a maximal cell $\mathcal C^+$ (resp. $\mathcal C^-$) in $\mathfrak S_S$ which is the intersection of all positive (resp. negative) open half spaces in the hyperplane arrangement in \cref{hyperplanearrangements}. Given $\mathfrak D$ a consistent $\mathfrak g$-SD, we define
\[
	\mathfrak p_{+,-}(\mathfrak D) \colon = \mathfrak p_\gamma (\mathfrak D)
\]
for any $\mathfrak S_S$-path $\gamma$ with $\gamma(0)\in \mathcal C^+$ and $\gamma(1)\in \mathcal C^-$. It does not depend on $\gamma$ since $\mathfrak D$ is consistent. This defines a map $\mathfrak p_{+,-}$ from $\mathfrak g $-$\mathcal {SD}$ the set of all consistent $\mathfrak g$-SD's to $G = \exp(\mathfrak g)$ by sending $\mathfrak D$ to $\mathfrak p_{+,-}(\mathfrak D)$. The following theorem asserts that $\mathfrak p_{+,-}$ is a bijection.
 
\begin{theorem}[cf. \cite{kontsevich2014wall}]\label{thesecondbijection}
The map $\mathfrak p_{+,-}\colon \mathfrak g\text{-}\mathcal {SD} \rightarrow G$, $\mathfrak D\mapsto \mathfrak p_{+,-}(\mathfrak D)$ is a bijection of sets. 
\end{theorem}

\begin{proof}
	Let $g\in G$. We construct a $\mathfrak g$-SD $\mathfrak D_g$ as follows. Recall the function in (\ref{thefunction})
	\[
	\phi_g\colon M_\mathbb R\rightarrow G,\ m\mapsto \pi_m(g).
	\] 
	Consider a $\mathfrak g$-SD $\mathfrak D_g = (\mathfrak S_S, \phi_{\mathfrak D_g})$ with $\phi_{\mathfrak D_g}(\sigma) = \phi_g(m)$ for any $\sigma\in \mathfrak S_S$ and $m\in \sigma^\circ$. We show it is consistent. In fact, a wall-crossing, or more generally face-crossing has the following description. Let $m^+\in \sigma^+$ and $m^-\in \sigma^-$. By the definitions of $\sigma^{\pm}$ and the uniqueness in \cref{factorization}, we have
	\[
	\pi_{m^+,+}(g) = \pi_{m,+}(g)\cdot \pi_m(g),\ \pi_{m^+}(g) = 1,\ \mathrm{and}\ \pi_{m^+, -}(g) = \pi_{m, -}(g);
	\]
	\[
	\pi_{m^-,+}(g) = \pi_{m,+}(g),\ \pi_{m^-}(g) = 1,\ \mathrm{and}\ \pi_{m^-, -}(g) =\pi_m(g)\cdot \pi_{m, -}(g).
	\]
	This gives
	\[
	\phi_g(m) = \pi_{m^-,+}^{-1}(g)\cdot \pi_{m^+, +}(g)\quad \mathrm{and}\quad \phi_g(m)^{-1} = \pi_{m^+,+}^{-1}(g)\cdot \pi_{m^-, +}(g).
	\]
By induction on the number of crossings, we have
\begin{equation}\label{pathorderedproductformula}
\mathfrak p_\gamma = \pi_{\gamma(1),+}^{-1}(g)\cdot \pi_{\gamma(0),+}(g) 	
\end{equation} 
for any $\mathfrak S_S$-path $\gamma$, which only depends on the end points, proving the consistency. Note that by construction $\mathfrak p_{+,-}(\mathfrak D_g) = \phi_g(0) = g$. This shows the map $\mathfrak p_{+,-}$ is surjective.

We show next the map $\mathfrak p_{+,-}$ is also injective. Let $\mathfrak D\in \mathfrak g\text{-}\mathcal{SD}$, i.e. a consistent $\mathfrak g$-SD. Let $\sigma\in \mathfrak S_S$ and we choose $m_\sigma \in \sigma^\circ$ and $\lambda\in \mathcal C^+$. Consider the path $\gamma\colon (-\infty,+\infty)\rightarrow M_\mathbb R$ given by $\gamma(t)=m_\sigma - \lambda t$. Whenever $\gamma$ meets some cone $\tau^\circ\in \mathfrak S_S^\circ$, it always goes from $\tau^+$ to $\tau^-$. Thus after rescaling, we get a $\mathfrak S_S$-path $\gamma$ going from $\mathcal C^+$ to $\mathcal C^-$ with positive crossings at $\sigma_1,\cdots,\sigma_k$ in order where $\sigma = \sigma_l$ for some $l$. The path-ordered product is then
\[
\mathfrak p_{+,-}(\mathfrak D) = \phi_{\mathfrak D}(\sigma_k)\cdots \phi_{\mathfrak D}(\sigma_l) \cdots \phi_{\mathfrak D}(\sigma_1)
\]
If $i<l$, then $m_i\colon =m_\sigma + t_i\cdot m^+\in \sigma _i^\circ$ for some $t_i>0$. If $n\in P_{m_i, 0}$, i.e. $\langle m_i, n\rangle = 0$, then $\langle m_\sigma + t_i\cdot m^+, n\rangle=0$ which implies $\langle m_\sigma, n\rangle<0$, i.e. $n\in P_{m_\sigma, -}$. Thus we have $P_{m_i,0}\subset P_{m_\sigma, -}$ and $\phi_{\mathfrak D}(\sigma_i)\in G_{m_{\sigma},-}$. Similarly if $j>l$, we have $P_{m_j,0}\subset P_{m_\sigma, +}$ and $\phi_{\mathfrak D}(\sigma_j)\in G_{m_{\sigma},+}$. Thus we have,
\[
\mathfrak p_{+,-}(\mathfrak D) = \left(\phi_{\mathfrak D}(\sigma_k)\cdots \phi_{\mathfrak D}(\sigma_{l+1})\right)\cdot\phi_{\mathfrak D}(\sigma)\cdot \left(\phi_{\mathfrak D}(\sigma_{l-1})\cdots \phi_{\mathfrak D}(\sigma_{1})\right)
\]
as a factorization with respect to $m$ in \cref{factorization}. This in particular shows 
\[
\phi_\mathfrak D(\sigma) = \pi_{m_\sigma}(\mathfrak p_{+,-}(\mathfrak D)) 
\]
and therefore the consistent $\mathfrak g$-SD is completely determined by $\mathfrak p_{+,-}(\mathfrak D)\in G$, i.e. the map $\mathfrak p_{+,-}$ is injective. 
\end{proof}

We clarify in the next proposition the relation between the notion of a $\mathfrak g$-WCS and that of a consistent $\mathfrak g$-SD when $\mathrm{Supp}(\mathfrak g)$ is finite. It is an immediate consequence of the proof \cref{thesecondbijection}.

\begin{proposition}\label{equivalenceoftwodefinitions}
Let $g\in G$ and $\mathfrak D_g = (\mathfrak S_S, \phi_{\mathfrak D_g})$ be $\mathfrak p_{+,-}^{-1}(g)\in \mathfrak g\text{-}\mathcal {SD}$. By \cref{bijection}, there is also  a $\mathfrak g$-WCS $s_g$ and it is determined by the function $\phi_g\colon M_\mathbb R\rightarrow G$ in (\ref{thefunction}). Then for any $\sigma\in \mathfrak S_S$, we have $$\phi_{\mathfrak D_g} (\sigma) = \phi_g (m)$$ for any $m\in \sigma^\circ$. 	
\end{proposition}

According to the above proposition, we will simply denote the $\mathfrak g$-SD $\mathfrak p^{-1}_{+,-}(g)$ by $$\mathfrak D_g = (\mathfrak S_S, \phi_g)$$ for $g\in G$. The function $\phi_g$ then has domain $M_\mathbb R$ rather than $\mathfrak S_S$. However by $\phi_g(\sigma)$, we always mean $\phi_{\mathfrak D_g}(\sigma)$ without any ambiguity.

\subsubsection{}  A consistent $\mathfrak g$-SD $\mathfrak D$ is determined by its wall-crossings among all the crossings. Let $\sigma \in \mathfrak S_S$ and suppose codim $\sigma >1$. One can always find a $\mathfrak S_S$-path $\gamma$ from $\sigma^+$ to $\sigma^-$ with only wall-crossings. Thus the face-crossing $\phi_\mathfrak D(\sigma)$ is equal to a product of wall-crossings because of consistency.  

Deonte the subset of all cones of codimension one in $\mathfrak S_S$ by $\mathfrak W_S$. The following is yet another equivalent definition to \cref{consistentscatteringdiagram}.

\begin{definition}
	Let $\mathfrak g$ be an $N^+$-graded Lie algebra with finite support $S$. A consistent $\mathfrak g$-SD $\mathfrak D$ is the datum of a function $\phi_\mathfrak D\colon \mathfrak W_S \rightarrow G$ such that 
	\begin{enumerate}[label=(\roman*)]
		\item $\phi(\sigma)\in \exp(\mathfrak g_\sigma)\subset G$ for any $\sigma\in \mathfrak W_S$;
		\item any path-ordered product only depends on end points. Here we only allow $\mathfrak S_S$-paths with wall-crossings (instead of face-crossings).
	\end{enumerate}
\end{definition}

\subsubsection{} Now we remove the finiteness restriction of $S = \mathrm{Supp}(\mathfrak g)$ to define consistent $\mathfrak g$-SDs in general. For $I\in \mathrm{Cofin}(N^+)$, the quotient Lie algebra $\mathfrak g^{<I}$ is supported on $S^{<I} = S\setminus I$. Recall that we have a group homomorphism $\rho^{<I}\colon \hat{G}\rightarrow G^{<I}$.

\begin{definition}\label{sdwithinfinitesupport}
A $\mathfrak g$-SD $\mathfrak D$ is a function $\phi_\mathfrak D \colon M_\mathbb R\rightarrow \hat G$ such that the induced function $$\phi_\mathfrak D^{<I} = \rho^{<I}\circ \phi_{\mathfrak D}\colon M_\mathbb R\rightarrow G^{<I}$$ for each $I\in \mathrm{Cofin}(N^+)$ comes from a $\mathfrak g^{<I}$-SD $\mathfrak D^{<I} = (\mathfrak S_{S^{<I}}, \phi_{\mathfrak D^{<I}})$, i.e.
\[
	\phi_{\mathfrak D}^{<I}(m) = \phi_{\mathfrak D^{<I}}(\sigma)
\]
for any $\sigma\in \mathfrak S_{S^{<I}}$ and any $m\in \sigma^\circ$.
It is said to be consistent if every $\mathfrak D^{<I}$ is consistent.
\end{definition}

\begin{proposition}
	The set $\mathfrak g$-$\mathcal{SD}$ of all consistent $\mathfrak g$-SDs is in bijection with $\hat G$ by sending $\mathfrak D$ to $\phi_{\mathfrak D}(0)\in \hat G$.
\end{proposition}

\begin{proof}
	Suppose $\mathfrak D$ is a consistent $\mathfrak g$-SD and let $g = \phi_\mathfrak D(0)$. Then by definition, we have $\phi_\mathfrak D^{<I}(0) = \rho^{<I} (g)$ for any $I\in \mathrm{Cofin}(N^+)$. Since $\mathfrak D^{<I}$ is consistent, for any $m\in M_\mathbb R$, $$\rho^{<I}(\phi_\mathfrak D(m))= \phi_{\mathfrak D^{<I}}(m) = \pi_{m}(\rho^{<I}(g)) = \rho^{<I}(\pi_m(g))$$ for any $I$. This shows $\phi_\mathfrak D(m) = \pi_m(g)$ for any $m\in M_\mathbb R$. Such a function $\phi_\mathfrak D\colon M_\mathbb R\rightarrow \hat G$, $m\mapsto \pi_m(g)$ obviously defines a consistent $\mathfrak g$-SD of \cref{sdwithinfinitesupport} for arbitrary $g\in \hat G$. The wanted bijection then follows.
\end{proof}

\subsection{Minimal support}\label{conecomplex}
When the Lie algebra $\mathfrak g$ has finite support $S$, usually the cone complex $\mathfrak S_S$ induced by the arrangement of hyperplanes is too fine to describe the scattering diagram $\mathfrak D_g = (\mathfrak S_S, \phi_g)$. For example, a path-connected component of $\phi_g^{-1}(h)$ for some $h\in G$ may be a union of cones in $\mathfrak S_S^\circ$. The following theorem is our main result in this section, which gives a minimal choice of the underlying cone complex.

\begin{theorem}\label{minimalsupport}
For any $g\in G$, the associated map $\phi_g\colon M_\mathbb R\rightarrow G$ satisfies the following properties. 
\begin{enumerate}[label = (\roman*)]
	\item For any $h\in G$, the preimage $\phi_g^{-1}(h)$ is relatively open in some subspace of $M_\mathbb R$.
	\item Each connected component of $\phi_g^{-1}(h)$ is the relative interior of a rational polyhedral cone.
	\item These cones form a finite complete cone complex $\mathfrak S_g$ of $M_\mathbb R$.
\end{enumerate}

\end{theorem}

\begin{remark}
In the literature, for example in \cite{kontsevich2014wall, bridgeland2016scattering, gross2018canonical}, the codimension one skeleton of $\mathfrak S_g$ is corresponding to the \textit{minimal or essential support} and codimension one cones are usually referred to as \textit{walls}.
\end{remark}

\subsubsection{} We need some preparations before the proof to \cref{minimalsupport}. Assume that $S = \mathrm {Supp} (\mathfrak g)$ is finite.

\begin{definition}
Let $g\in G$. We define $\mathrm{Supp}(g)$ to be the minimal subset of $\mathrm{Supp}(\mathfrak g)$ such that $\mathfrak g_{\mathrm{Supp}(g)}$ is a Lie subalgebra of $\mathfrak g$ that contains $\log (g)$. We say $n\in \mathrm{Supp}(g)$ is \textit{extreme} if it is not a positive linear combination of other elements in $\mathrm{Supp}(g)$. Denote the subset of all extreme elements by $E(g)$.
\end{definition}

We know from \cref{equivalenceoftwodefinitions} that the function $\phi_g$ is constant on any $\sigma^\circ \in \mathfrak S_S^\circ$.

\begin{lemma}\label{relativeopen}
For any $h\in G$, the preimage $\phi_g^{-1}(h)$, as a union of cones in $\mathfrak S_S^\circ$, is relatively open in the subspace $\mathrm{Supp}(h)^\perp \subset M_\mathbb R$. 
\end{lemma}
\begin{proof}
First of all the set $\phi_g^{-1}(h)$ is a union of cones in $\mathfrak S_S^\circ$ as $\phi_g$ is constant on any $\sigma^\circ$. We have $\phi^{-1}_g(h)\subset \mathrm{Supp}(h)^\perp$ since $\log(h)$ is supported on $\mathrm{Supp}(h)$. We just need to show that for any $\sigma^\circ$ contained in $\phi_g^{-1}(h)$,  \[|\mathrm{Star}(\sigma)|\cap {\mathrm{Supp}(h)^\perp}\subset \phi_g^{-1}(h).\]
Let $m\in \sigma^\circ $ and $m'\in |\mathrm{Star}(\sigma)|\cap {\mathrm{Supp}(h)^\perp}.$ By \cref{keyfunctor}, $$\phi_g(m') = \pi_{m'}(g) = \pi_{m,m'}(\pi_{m}(g)) = \pi_{m,m'}(h).$$
The map $\pi_{m,m'}$ depends on the partition
\[
S_{m,0} = (S_{m,0})_{m',+}\sqcup S_{m',0}\sqcup (S_{m,0})_{m',-}.
\]
Note that by assumption we have $\mathrm{Supp}(h)\subset S_{m',0}\subset S_{m, 0}$. Therefore $\phi_g(m') = \pi_{m,m'}(h) = h$ which implies any such $m'$ is contained in $\phi_g^{-1}(h)$. This finishes the proof.
\end{proof} 

\begin{lemma}\label{supportoflink}
Let $\sigma\in \mathfrak S_S$. Suppose that $e\in E(\phi_g(\sigma))$ is extreme. Then for any cone $\rho$ in $\mathfrak S_S$ that is contained in $e^\perp$ and contains $\sigma$, we have the component $\log(\phi_g(\rho))_e \neq 0$, i.e. $e\in \mathrm{Supp}(\phi_g(\rho))$. 
\end{lemma}
\begin{proof}
Let $h = \phi_g(\sigma)$. By \cref{keyfunctor}, we have $\phi_g(\rho) = \pi_{\sigma, \rho} (h)$ and
\[
h = \pi_{m,+}(h)\cdot \phi_g(\rho)\cdot \pi_{m,-}(h)
\]
where $m\in \rho^\circ$. Note that $e\in S_{\rho,0}\subset S_{\sigma, 0}$ and $\pi_{m,\pm}(h)$ are supported outside of $S_{\rho, 0}$. Since $e$ is extreme, it is an immediate consequence of the Baker–Campbell–Hausdorff formula that $\log(\phi_g(\rho))_e \neq 0$.
\end{proof}

\begin{proof}[Proof of \cref{minimalsupport}]
Part (i) follows from \cref{relativeopen}. Let $V$ be a connected component of $\phi_g^{-1}(h)$. We know that the closure $\overline V$ is a union of cones in $\mathfrak S_S$. The convexity of $\overline V$ can be proved locally. In fact, it suffices to prove that for any $\sigma\in \mathfrak S_{S}$ in the boundary $\overline V\setminus V$, there exists some $e\in N$ such that any connected component $\tau$ of $|\mathrm{Star}(\sigma)|\cap V$ is contained in one of the two open halves of $\mathrm{Supp}(h)^\perp$ separated by the hyperplane $e^\perp$ which passes through $\sigma$. We prove this by finding some $e$ such that for any $\rho\subset \mathrm{Supp}(h)^\perp \cap e^\perp$ that contains $\sigma$, $\phi_g({\rho})\neq h$, which implies $\tau$ cannot cross the hyperplane $e^\perp$. In fact, since $\sigma\cap V = \emptyset$, $\phi_g(\sigma)\neq \phi_g(\tau) = h$. We just take some extreme element $e$ in $(S_{\sigma,0}\setminus S_{\tau,0})\cap E(\phi_g(\sigma))$. It is clear that $e^\perp \cap \mathrm{Supp}(h)^\perp$ is a hyperplane in $\mathrm{Supp}(h)^\perp$ and we have $e\in \mathrm{Supp}(\phi_g(\rho))$ by \cref{supportoflink}. Therefore $\phi_g(\rho)\neq h$ since $e\notin \mathrm{Supp}(h)$. This proves part (ii).

Let $\delta$ be a face of $\overline {V}$. It is a union of cones in $\mathfrak S_S$. The map $\phi_g$ is constant on the relative interior of $\delta$,  otherwise $\overline{V}$ would be split into two components. Suppose that $\phi_g$ remains constant on some face $\sigma'$ of $\sigma$. By \cref{relativeopen}, it remains constant on $|\mathrm{Star}(\sigma')|\cap \langle\sigma\rangle_\mathbb R$. This means $\phi_g$ also extends constantly from $V$ to some face of $\overline{V}$ that contains $\sigma'$, which contradicts with the assumption of $V$ being a connected component of $\phi_g^{-1}(h)$. Therefore we conclude that the interior of any face of $\overline{V}$ is also a connected component of some preimage of $\phi_g$. This proves part (iii) that the set of all cones of the form $\overline{V}$ form a complete finite cone complex in $M_\mathbb R$.
\end{proof}

\subsubsection{} Now we remove the finiteness restriction of $\mathrm{Supp}(\mathfrak g)$. Fix $g\in \hat G$ and consider the corresponding consistent $\mathfrak g$-SD $\mathfrak D_g$ and the function $\phi_g\colon M_\mathbb R\rightarrow \hat G$.

\begin{proposition}\label{pathconnectedcomponent}
	For any $h\in \hat G$, each path-connected component of the preimage $\phi_g^{-1}(h)$ is a convex cone in $M_\mathbb R$. These cones together give a decomposition of $M_\mathbb R$. 
\end{proposition}

\begin{proof} For each $I\in \mathrm{Cofin}(N^+)$, the projection $g^{<I}$ of $g$ in $G^{<I}$ determines a consistent $\mathfrak g^{<I}$-SD. The induced cone complex $\mathfrak S_{g^{<I}}$ is a refinement of $\mathfrak S_{g^{<J}}$ if $I\subset J$. Thus we have a directed inverse system of complete finite cone complexes $\mathfrak S_{g^{<I}}$ with indexed by $\mathrm{Cofin}(N^+)$. This gives a decomposition of $M_\mathbb R$ in the limit into convex cones as follows. Let $C$ be a path-connected component of $\phi_g^{-1}(h)$ and $m\in C$. We define
\begin{equation}\label{limitconecomplex}
\sigma_m^\circ \colon = \bigcap_{I\in \mathrm{Cofin}(N^+)} (\sigma_m^{<I})^\circ
\end{equation}
where $\sigma_m^{<I}$ is the unique cone in $\mathfrak S_{g^{<I}}$ whose relative interior contains $m$. Thus $\sigma_m^\circ$ is still a convex cone (but may no longer be the interior of a rational polyhedral cone) and in particular is path-connected. The map $\phi_g\colon M_\mathbb R\rightarrow \hat G$ is obviously constant on $\sigma_m^\circ$ so $\sigma_m^\circ \subset C$. On the other hand,  we have $C\subset (\sigma_m^{<I})^\circ$ for any $I$ since $(\sigma_m^{<I})^\circ$ is the connected component of $\phi_{g^{<I}}^{-1}(h^{<I})$. Then we have $C = \sigma_m^\circ$. Therefore $M_\mathbb R$ is a disjoint union of convex cones of the form $\sigma_m^\circ$. 
\end{proof}

The collection of cones in \cref{pathconnectedcomponent} is now denoted by
\begin{equation}\label{scatteringdecomposition}
    \mathfrak S_g^\circ \colon = \{\sigma_m^\circ\mid m\in M_\mathbb R\}.
\end{equation}
It is a decomposition of the space $M_\mathbb R$ into convex cones. Define $\sigma_m\colon = \overline {\sigma_m^\circ}$. We also the following set of closed cones
\[
\mathfrak S_g \colon =\{\sigma_m\mid m\in M_\mathbb R\}.
\]
Note that in the case of finite support, these two collection of cones coincide with our cone complex $\mathfrak S_g$ in \cref{minimalsupport} and the induced cone decomposition by taking the relative interiors. Here we emphasize that $\mathfrak S_g$ (and thus $\mathfrak S_g^\circ$) also has a poset structure as the projective limit of directed inverse system of posets $\mathfrak S_{g^{<I}}$, i.e.
\[
	\sigma_{m_1}\prec \sigma_{m_2}\Leftrightarrow \sigma_{m_1}\subset \sigma_{m_2}.
\]
However the cones $\sigma_m$ may no longer be rational polyhedral. One should view $\mathfrak S_g$ not only as a set of cones but also recording the cone $\sigma_m^\circ$ for every $\sigma_m$. Note that since $\sigma_m$ may no longer be rational polyhedral, it has no relative interior in general. Therefore in the support-infinite case, $\mathfrak S_g$ is more sophisticated than a finite complete cone complex. But we will still call $\mathfrak S_g$ a cone complex sometimes.

\begin{proposition}\label{rationalconeandfaces}
Let $g\in \hat G$. Suppose $\sigma\in \mathfrak S_g$ is a rational polyhedral cone and it appears in $\mathfrak S_{g^{<I}}$ for some $I$. Then all the faces of $\sigma$ are elements of $\mathfrak S_g$.
\end{proposition}
\begin{proof}
	By \cref{minimalsupport}, $\mathfrak S_{g^{<I}}$ is a cone complex so the faces of $\sigma$ are cones in $\mathfrak S_{g^{<I}}$. Since $\sigma$ is in $ \mathfrak S_g$, $\sigma$ belongs to $\mathfrak S_{g^{<J}}$ for any $J\subset I$ and so do its faces. Therefore the faces are also elements in the projective limit $\mathfrak S_g$.   
\end{proof}

\begin{remark}[Consistent $\mathfrak g$-SD updated]\label{scatteringdiagramnewdef}
	Let $\mathfrak g$ be an $N^+$-graded Lie algebra and $g$ be an element in the corresponding pro-unipotent group $\hat G$. The \textit{consistent} $\mathfrak g$\textit{-SD} corresponding to the group element $g\in \hat G$ now refers to the data $\mathfrak D_g = (\mathfrak S_g,\phi_g)$ consisting of the collection of cones $\mathfrak S_g$ (\ref{scatteringdecomposition}) and the function $\phi_g\colon M_\mathbb R\rightarrow \hat G$ (\ref{thefunction}) which is constant along each cone in $\mathfrak S_g^\circ$. 
\end{remark}

\subsubsection{}\label{mapsbetweensds} We set up some conventions that will be useful later. Let $f\colon \mathfrak g_1\rightarrow \mathfrak g_2$ be a homomorphism of $N^+$-graded Lie algebras. It induces a group homomorphism $F\colon \hat G_1\rightarrow \hat G_2$ and consequently a map
\[
F\colon \mathfrak g_1\text{-}\mathcal {SD}\rightarrow \mathfrak g_2\text{-}\mathcal {SD}, \quad \mathfrak D_g\mapsto \mathfrak D_{F(g)}
\]
for $g\in \hat G_1$. Clearly we have 
\[
\phi_{F(g)} = F\circ \phi_g.
\]

\subsection{Consistency revisited} In this section, we explain the consistency in terms of path-ordered products using $\mathfrak S_g$.
\subsubsection{} First we assume $S = \mathrm{Supp}(\mathfrak g)$ is finite. Then \cref{minimalsupport} applies. In particular, the cone complex $\mathfrak S_g$ is a coarsening of $\mathfrak S_S$. For each $\tau$  in $\mathfrak S_g$ of codimension at least one, there is a relatively positive maximal cell $\tau^+$ in $\mathfrak S_g$ incident to $\tau$ as described in \cref{consistent}. Similarly there is a relatively negative maximal cell $\tau^-$. Then we define $\mathfrak S_g$-paths and path-ordered products with respect to $\mathfrak S_g$ exactly the same way as for $\mathfrak S_S$ in \cref{consistent}. Then the following proposition follows directly from in \cref{thesecondbijection} the consistency of $\mathfrak D_g$ (in terms of $\mathfrak S_S$ in \cref{consistentscatteringdiagram}).

\begin{proposition}\label{consistency}
Let $\gamma$ be a $\mathfrak S_g$-path. Then we have 
\[
\mathfrak p_\gamma (\mathfrak D_g) = \pi_{\gamma(1),+}^{-1}(g)\cdot \pi_{\gamma(0),+}(g).
\]
In particular, it only depends on the end points $\gamma(0)$ and $\gamma (1)$.  
\end{proposition}

\subsubsection{} Now we do not assume that $\mathrm{Supp}(\mathfrak g)$ is finite. One wishes to define path-ordered products even though the collection $\mathfrak S_g$ may be infinite.

\begin{definition}\label{pathandpathorderedproduct}
A smooth curve $\gamma\colon [0,1]\rightarrow M_\mathbb R$ is said to be a $\mathfrak S_g$-\textit{path} if it is a $\mathfrak S_{g^{<I}}$-path (\cref{path}) for any cofinite ideal $I\subset N^+$.  
\end{definition}

\begin{lemma}
	Let $\mathfrak D_g = (\mathfrak S_g, \phi_g)$ be the consistent $\mathfrak g$-SD corresponding to $g\in \hat G$ and $\gamma$ be a $\mathfrak S_g$-path. Then for any $I\subset J$, we have
	\[
	\rho_{I,J}(\mathfrak p_{\gamma}(\mathfrak D_{g^{<I}})) = \mathfrak p_{\gamma}(\mathfrak D_{g^{<J}}).
	\]
	
\end{lemma}

\begin{proof}
	By definition, $\gamma$ is also a $\mathfrak D_{g^{<I}}$-path for any $I\in \mathrm {Cofin}(N^+)$. By \cref{consistency}, it amounts to show that
	\[
	\rho_{I,J}\left(\pi_{\gamma(1),+}^{-1}(g^{<I})\cdot \pi_{\gamma(0),+}(g^{<I})\right) = \pi_{\gamma(1),+}^{-1}(g^{<J})\cdot \pi_{\gamma(0),+}(g^{<J}).
	\]
	This follows from the fact that $\pi_{m,+}$ commutes with $\rho_{I,J}$.
\end{proof}

The above lemma allows us to propose the following definition of the path-ordered product for a $\mathfrak S_g$-path $\gamma$ in general.
\begin{definition}
The \textit{path-ordered product} $\mathfrak p_\gamma (\mathfrak D_g)$ of $\gamma$ is defined to be the projective limit in $\hat G$ of the path-ordered products $\mathfrak p_{\gamma}(\mathfrak D_{g^{<I}})$ for $I\in \mathrm{Cofin}(N^+)$, i.e.
\[
\mathfrak p_\gamma (\mathfrak D_g)\colon = \lim_{\underset{I}{\longleftarrow}}\mathfrak p_{\gamma}(\mathfrak D_{g^{<I}}).
\]
\end{definition}

The consistency of $\mathfrak p_\gamma (\mathfrak D_g)$ follows directly from the definition. The following is the support-infinite version of \cref{consistency}.
\begin{proposition}
	For any $\mathfrak D_g$-path $\gamma$, we have
	\[
	\mathfrak p_\gamma (\mathfrak D_g) = \pi_{\gamma(1),+}^{-1}(g)\cdot \pi_{\gamma(0),+}(g).
	\]
	In particular, it only depends on the end points $\gamma(0)$ and $\gamma (1)$.
\end{proposition}

%% file: 3_cluster.tex
In this section, we first review the cluster scattering diagrams defined in \cite{gross2018canonical} and their quantum analogs. Then we review the Hall algebra scattering diagrams and the stability scattering diagrams defined in \cite{bridgeland2016scattering}. 

\subsection{Initial data}\label{indata}

As in \cref{setting}, we have $N$ a lattice of rank $n$, but now equipped with a $\mathbb Z$-valued skew-symmetric form
\[
\{\ ,\ \}\colon N\times N\rightarrow \mathbb Z.
\]

\begin{definition}\label{skew}
An $N$-graded Lie algebra $\mathfrak g$ is \textit{skew-symmetric} if $\{d_1,d_2\}=0$ implies $[\mathfrak g_{d_1},\mathfrak g_{d_2}]=0$ for any $d_1, d_2$ in $N$.
\end{definition}

This feature ensures that the so-called initial data is able to determine a $\mathfrak g$-scattering diagram (see \cite[section 3]{kontsevich2014wall}). We briefly review this important point of view here. Recall that a consistent $\mathfrak g$-SD is determined by an element in $\hat{G}$. We introduce another way to parametrize elements in $\hat G$ as in \cite[section 3]{kontsevich2014wall} and also in \cite[section 1.2]{gross2018canonical}. Define a linear map $p^*\colon N\rightarrow M$ by $n\mapsto \{n,\cdot\}\in M$. Note that we do not require that the map $p^*$ to be an isomorphism. Let $n$ be a primitive element in $N^+$. Consider the decomposition of $\mathfrak g$ with respect to $p^*(n)$ as in \cref{factorization}
\begin{equation}\label{decomposition}
\mathfrak g = \mathfrak g_{p^*(n),+}\oplus \mathfrak g_{p^*(n),0}\oplus\mathfrak g_{p^*(n),-}    
\end{equation}
which gives a factorization of the corresponding pro-nilpotent group $\hat G = \hat G_{p^*(n),+}\cdot  \hat G_{p^*(n),0}\cdot  \hat G_{p^*(n),-}$. The Lie subalgebra $\mathfrak g_{p^*(n),0}$ further decomposes into 
\[
\mathfrak g_{p^*(n),0} = \mathfrak g_{n}^{||}\oplus \mathfrak g_{p^*(n),0}^{\perp}
\]
where $\mathfrak g_{n}^{||} \colon = \mathfrak g_{\mathbb Z^+n}$. Note that the Lie subalgebra $\mathfrak g_{n}^{||}$ is central in $\mathfrak g_{p^*(n),0}$ and $\mathfrak g_{p^*(n),0}^\perp$ is an ideal of $\mathfrak g_{p^*(n),0}$. This gives a group homomorphism
\[
r_n\colon \hat{G}_{p^*(n),0}\rightarrow \hat{G}_{n}^{||}.
\]
Given an element $g \in \hat G$, for each primitive $n \in N^+$, we define $$\psi_n(g) \colon = r_n(g_{p^*(n),0})\in \hat{G}_n^{||}.$$ This defines a map (of sets)
\begin{equation}\label{initialdata}
\psi \colon \hat G\rightarrow \prod_{\mathrm{primitive}\ n\in N^+}\hat{G}_n^{||} = \prod_{d\in N^+} \mathfrak g_d,\quad \psi (g) = (\psi_n(g))_{\mathrm{primitive}\ n\in N^+}.   
\end{equation}

\begin{proposition}[\cite{kontsevich2014wall}, proposition 3.3.2]
The map $\psi$ is a bijection of sets.   
\end{proposition}

This proposition provides another way (other than the bijection $\hat G = \prod_{d\in N^+} \mathfrak g_d$ in \ref{socfi}) to express an element $g$ in $\hat G$ by its components in each $\mathfrak g_d$ under $\psi$. This expression of $g$ is called the \textit{initial data} of the corresponding consistent $\mathfrak g$-SD. The fact that a consistent $\mathfrak g$-SD is determined by its initial data is sometimes known as ``a consistent scattering diagram is determined by its incoming walls'', e.g. see \cite{gross2018canonical}.

\subsection{Cluster scattering diagrams and quantization}\label{clustersdsetting}
\subsubsection{} Consider the $N$-graded Poisson algebra (usually called the torus Lie algebra) defined as follows:
\[T\colon = \mathbb Q[N] = \bigoplus _{d\in N} \mathbb Q \cdot x^d\quad \mathrm{and}\quad [x^{d_1},x^{d_2}]=\{d_1,d_2\}x^{d_1+d_2},\ d_1,d_2\in N.
\]
Let $\mathbf{s}$ be a basis of $N$. The subsemigroup $N^+_{\mathbf{s}}$ is defined as in \cref{setting}. Define the Lie algebra \[\mathfrak g_\mathbf{s}\colon = T_{N^+_\mathbf{s}} \subset T.\] It is clearly $N^+_\mathbf{s}$-graded and skewsymmetric in the sense of \cref{skew}. Now take $\mathfrak g$ to be $\mathfrak g_\mathbf{s}$. \textit{The dilogarithm series} $\mathrm{Li}_2(x)\in \mathbb Q[[x]]$ is defined by
\begin{equation}\label{dilogarithm}
    \mathrm{Li}_2(x)\colon = \sum_{k=1}^\infty \frac{x^{k}}{k^2}.
\end{equation}
We define the following element in $\hat G$ under the identification $\psi$ in  (\ref{initialdata})
\begin{equation}\label{initialelement}
g_\mathbf{s}\colon = (g_{n})_{n\ \mathrm{primitive}}\in \prod_{n\ \mathrm{primitive}} \hat G_n^{||}    
\end{equation}
where
\[
g_{s_i} = \exp\left(-\mathrm{Li}_2(-x^{s_i})\right)  = \exp\left(\sum_{k=1}^\infty \frac{(-1)^{k-1}x^{ks_i}}{k^2}\right) \in \hat{G}_{s_i}^{||}
\]
and $g_n= \mathrm{id}$ for any other primitive $n$. 

\begin{definition}[\cite{gross2018canonical}]\label{clustersd}
The \textit{cluster scattering diagram} $\mathfrak D_\mathbf{s} = (\mathfrak S_\mathbf s, \phi_\mathbf s)$ is defined to be 
\[
\mathfrak D_{\mathfrak g_\mathbf s}= (\mathfrak S_{g_\mathbf s}, \phi_{g_\mathbf s})
\]
i.e. the consistent $\mathfrak g_\mathbf{s}$-SD corresponding to the element $g_\mathbf{s}\in \hat {G}$ (\cref{scatteringdiagramnewdef}).
\end{definition}

\subsubsection{} The $N$-graded poisson algebra $T$ has a quantization as follows. Let $T_q\colon  = \mathbb Q(q^{\frac{1}{2}})[N]$. With the twisted product $*$ defined by
\begin{equation}\label{twistedprod}
x^{d_1} * x^{d_2} = q^{\frac{1}{2}\{d_1,d_2\}} x^{d_1+d_2},    
\end{equation}
the $\mathbb Q(q^\frac{1}{2})$-vector space $T_q$ becomes an associative algebra over $\mathbb Q(q^\frac{1}{2})$. Let $R$ be any $\mathbb Q[q^{\pm\frac{1}{2}}]$-subalgebra of $\mathbb Q(q^\frac{1}{2})$ such that both $q^\frac{1}{2}-1$ and $q^\frac{1}{2}+1$ are non-invertible. This means we can evaluate any element in $R$ at $q^\frac{1}{2} = 1$ or $-1$. Take the subalgebra
\[
R[N]\subset T_q = \mathbb Q(q^\frac{1}{2})[N].
\]
In fact $R[N]$ is an $N$-graded Poisson algebra under the bracket
\[
\{x^{d_1},x^{d_2}\} \colon = \frac{[x^{d_1},x^{d_2}]}{q^{\frac{1}{2}}-q^{\frac{-1}{2}}} \in R[N].
\]
There is an $N$-graded Poisson algebra homomorphism from $R[N]$ to $T$ by evaluating coefficients in $R$ at $q^{\frac{1}{2}}=1$ (getting another Poisson algebra structure on $T$ at $q^{\frac{1}{2}}=-1$).

Consider the normalized basis of $T_q$ given by \[\hat x^d = \frac{x^d}{q^{\frac{1}{2}}-q^{\frac{-1}{2}}}\] and the $R$-submodule $R[\hat x^d,d\in N]$ in $T_q$ generated by this basis. One can check this $R$-submodule is a skewsymmetric $N$-graded Lie algebra under the commutator bracket of $T_q$. In fact, the $R$-linear map
\[
R[\hat x^d,d\in N]\rightarrow R[N],\ \hat x^d\mapsto x^d
\]
is a Lie algebra homomorphism. To summarize we have the following commutative diagram (of Lie algebras).
\[
\begin{tikzcd}
R[N] \arrow[r,"q^{1/2}\mapsto1"]&T = \mathbb Q[N]\\
R[\hat x^d,d\in N] \arrow[u, "\hat x^d\mapsto x^d"]\arrow[ur, swap, "\hat x^d\mapsto x^d\ q^{1/2}\mapsto 1"]
\end{tikzcd}
\]

Now let $R$ be the algebra
\begin{equation}\label{regularalgebra}
R = \mathbb Q_\mathrm{reg}(q^\frac{1}{2})\colon = \mathbb Q[q^{\pm\frac{1}{2}}][(1+q+\cdots+q^{k})^{-1},k\geq 1]    
\end{equation}
and take 
\begin{equation}\label{quantumtorus}
\mathfrak g^q_{\mathbf{s}}\colon = R[\hat x^d, d\in N_\mathbf{s}^+].
\end{equation}
The later is an $N_\mathbf s^+$-graded Lie algebra with commutator bracket. We then have a group homomorphism 
\[
E \colon \exp (\hat{\mathfrak g}_{\mathbf s}^q)\rightarrow \exp (\hat{\mathfrak{g}}_\mathbf{s})
\]
induced by the Lie algebra homomorphism (evaluation at $q^{1/2} = 1$)

\begin{equation}\label{evaluationat1}
e = e_{q^{1/2}=1}\colon \mathfrak g_\mathbf s^q \rightarrow \mathfrak g_\mathbf s,\quad \hat x^d\mapsto x^d,\quad q^{1/2}\mapsto 1. 
\end{equation}

The group elements $g_i$'s in (\ref{dilogarithm}) can be lifted to $\hat {G} = \exp (\hat{\mathfrak g}_{\mathbf s}^q)$. Define 
\begin{equation}\label{quantumdilogarithm}
    g_{i}^q \colon = \exp\left(-\mathrm{Li}_2^q(-x^{s_i})\right)\colon = \exp \left( \sum_{k=1}^\infty \frac{(-1)^{k-1}\hat x^{ks_i}}{k[k]_q}\right)\in \hat{G}^{||}_{s_i}
\end{equation}
where $[k]_q = q^{\frac{k-1}{2}}+\cdots+ q^{-\frac{k-1}{2}}.$ Clearly we have $E(g_i^q) = g_i$ since $$e(-\mathrm{Li}_2^q(-x^{s_i})) = -\mathrm{Li}_2(-x^{s_i}).$$ Similarly, we define $g^q_{\mathbf{s}}$ to be the element in $\hat G$ such that under the identification $\psi$ in (\ref{initialdata}), it has components $\mathrm{id}\in \hat G^{||}_n$ for any primitive $n$ not in $\mathbf{s}$ and $g^q_{i}\in \hat{G}_{s_i}^{||}$ for each $s_i$. We also have
\[
E(g_\mathbf s^q) = g_\mathbf s.
\]

\begin{definition}\label{qcluster}
The \textit{quantum cluster scattering diagram} $\mathfrak D^q_\mathbf{s} = (\mathfrak S^q_\mathbf s, \phi^q_\mathbf s)$ is defined to be 
\[
\mathfrak D_{g_\mathbf s^q} = (\mathfrak S_{g_\mathbf s^q}, \phi_{g_\mathbf s^q})
\]
i.e. the consistent $\mathfrak g^q_\mathbf{s}$-SD corresponding to the element $g^q_\mathbf{s}\in \hat {G}$ (\cref{scatteringdiagramnewdef}).
\end{definition}

Clearly we have $\mathfrak D_{g_\mathbf s^q}$ is a quantization of $\mathfrak D_{g_\mathbf s}$, i.e.
\[
\mathfrak D_{g_\mathbf s} = E(\mathfrak D_{g_\mathbf s^q})
\]
in the language of \cref{mapsbetweensds}.

\begin{remark}\label{associativeembedding}
It is convenient to realize $g_{i}^q$ and $g_\mathbf s^q$ as elements in $\mathbb Q_\mathrm {}(q^{1/2})[[N_\mathbf s^{\oplus}]]$ where the completion is with respect to the grading. In fact, we have an embedding
\[
\exp  \colon \hat{\mathfrak g}^q_{\mathbf{s}} \hookrightarrow 
\mathbb Q(q^{1/2})[[N_\mathbf s^{\oplus}]]
\]
by taking exponentials using the twisted product $*$ in $\mathbb Q(q^{1/2})[[N_\mathbf s^{\oplus}]]$. One can show (using $q$-binomial theorem) that
\begin{equation}\label{expofquantumdilogarithm}
    g^q_{i} \colon = \exp\left(-\mathrm{Li}_2^q(-x^{s_i})\right)  =  \sum_{k=0}^\infty \frac{q^{k^2/2}x^{ks_i}}{[\mathrm{GL}_k]_q}
\end{equation}
where the later is called the \textit{quantum dilogarithm series} $\mathbb E(q^{1/2},x^{s_i})$ and $[\mathrm{GL}_k]_q$ counts the number of points of $\mathrm{GL}_k$ in $\mathbb F_q$.
\end{remark}

\subsection{Hall algebra and stability scattering diagrams}

In this section, we introduce the Hall algebra scattering diagram and the (quantum) stability scattering diagram for a quiver (seed) with potential following \cite{bridgeland2016scattering}.

\subsubsection{}\label{motivichallalgebra}  Let $\mathbf{s}=\{s_1, \dots, s_n\}$ be a basis of $N$ as before. There is a skewsymmetric integer matrix $B(\mathbf{s})$ defined by
\[
B(\mathbf{s})_{ij}= \{s_i,s_j\}
\]
which serves as the \textit{adjacency matrix} of a 2-acyclic quiver $Q=Q(\mathbf{s})$. Denote by $\widehat{\mathbb CQ}$ the completion of the path algebra $\mathbb CQ$ with respect to the ideal generated by paths of positive length. Let $\mathbf{w}\in \widehat{\mathbb CQ}$ be a potential, i.e. a formal sum of cyclic paths. We call the pair $\mathbf{(s,w)}$ a \textit{seed with potential}. The \textit{completed Jacobian algebra} of $(\mathbf{s,w})$ is
\[
J\colon =J(\mathbf s, \mathbf w)\colon = \widehat{\mathbb CQ}/\overline{\langle\partial \mathbf w\rangle}
\]
where $\partial \mathbf{w}\colon = \{\partial_\alpha \mathbf{w}\mid \alpha\in Q_1\}$ is the set of all cyclic derivatives of $\mathbf{w}$ and $\overline{\langle\partial \mathbf w\rangle}$ denotes the closure of the ideal of $\widehat{\mathbb CQ}$ generated by $\partial \mathbf{w}$; see \cite[definition 3.1]{derksen2008quivers}. We denote the category of finite dimensional left $J$-modules by $J$-mod.

Let $\mathfrak M(\mathbf{s,w})$ be the moduli stack of objects in $J$-mod; see \cite[section 6]{nagao2013donaldson} for a construction of this moduli stack. Note that the set of dimension vectors of $J$-mod is canonically identified with $N_\mathbf{s}^\oplus \colon = N_\mathbf{s}^+\cup \{0\}$. We have the following decomposition
\[
\mathfrak M\colon = \mathfrak M(\mathbf{s,w}) = \coprod_{d\in N_\mathbf{s}^\oplus} \mathfrak M_d(\mathbf{s,w}).
\]
where $\mathfrak M_d\colon = \mathfrak M_d(\mathbf{s,w})$ is the moduli stack of $J$-modules of fixed dimension vector $d$. When $d = 0$, the moduli stack $\mathfrak M_0$ is isomorphic to $\mathrm{Spec}(\mathbb C)$. 

We refer to \cite[section 5]{bridgeland2016scattering} and \cite[section 7]{nagao2013donaldson} for details of the following definitions. Let $K(\mathrm{St}/ \mathfrak M)$ be the \textit{relative Grothendieck group of stacks} over $\mathfrak M$. It is naturally a module over $K(\mathrm{St}/\mathbb C)$, the \textit{Grothendieck ring of stacks} over $\mathbb C$. Furthermore, one can define a convolution type product $*$ (\cite[theorem 4.1]{joyce2007configurations}) on $K(\mathrm{St}/\mathfrak M)$ so that it becomes an associative $K(\mathrm{St}/\mathbb C)$-algebra graded by $N_\mathbf{s}^\oplus$; see also \cite[5.4]{bridgeland2016scattering}.

There exists a ring homomorphism
\[
\Upsilon \colon K(\mathrm{St}/\mathbb C)\rightarrow \mathbb Q(q^{\frac{1}{2}})
\]
which takes a smooth projective complex variety $X$ to its Poincaré polynomial
\[
\sum_{k=1}^{2\dim_\mathbb C X}\dim_{\mathbb C}H_k(X_\mathrm{an},\mathbb C)(-q^{1/2})^k.
\]
We define the following $\mathbb Q(q^{\frac{1}{2}})$-vector space using $\Upsilon$:
\[
H(\mathbf{s,w})\colon = K(\mathrm{St}/\mathfrak M)\otimes_{K(\mathrm{St}/\mathbb C)}\mathbb Q (q^{\frac{1}{2}}).
\]
So $H(\mathbf{s,w})$ is then an associative algebra over $\mathbb Q(q^{\frac{1}{2}})$. This is \textit{the motivic Hall algebra} of the abelian category $J$-mod. It is $N^\oplus_\mathbf{s}$-graded. For each $d$, the subspace $H(\mathbf{s,w})_d$ is spanned by the elements of the form $[X\overset{f}{\rightarrow} \mathfrak M]$ which factor through the inclusion $\mathfrak M_d\hookrightarrow \mathfrak M$. Such an element will be simply denoted by $[X\overset{f}{\rightarrow} \mathfrak M_d]$. The unit of $H(\mathbf{s,w})$ is given by $[\mathfrak M_0\overset{\mathrm{id}}{\rightarrow} \mathfrak M_0]$. 

Taking the positively graded part, we have an $N^+_\mathbf{s}$-graded Lie algebra under the commutator bracket 
\[
{\mathfrak g}^{\mathrm{Hall}}_{\mathbf{s,w}}\colon=\bigoplus_{d\in N^+_\mathbf{s}}H(\mathbf{s,w})_d\subset H(\mathbf{s,w}).
\]
By taking the exponentials in the completion $\hat H(\mathbf{s,w})$, we have an embedding of the corresponding pro-unipotent group in the the Hall algebra
\[
\exp(\hat {{\mathfrak g}}^{\mathrm{Hall}}_{\mathbf{s,w}}) \cong  1+\hat{{\mathfrak g}}^{\mathrm{Hall}}_{\mathbf{s,w}}\subset \hat H(\mathbf{s,w}).
\]

\subsubsection{} The identity map from the moduli stack $\mathfrak M$ to itself defines an element in $\hat{H}(\mathbf{s,w})$ 
\[
1_\mathfrak M\colon = [\mathfrak M\rightarrow \mathfrak M]\in  1 + \hat {{\mathfrak g}}^{\mathrm{Hall}}_{\mathbf{s,w}}.
\]
It is in the group $\exp (\hat{{\mathfrak g}}^{\mathrm{Hall}}_{\mathbf{s,w}})$.

\begin{definition}
 The \textit{Hall algebra scattering diagram} $\mathfrak D_{\mathbf{s,w}}^{\mathrm{Hall}} = (\mathfrak S_\mathbf {s,w}^\mathrm{Hall}, \phi_\mathbf {s,w}^\mathrm{Hall})$ for $(\mathbf{s,w})$ is defined to be the consistent ${\mathfrak g}^{\mathrm{Hall}}_{\mathbf{s,w}}$-SD corresponding to the group element $1_{\mathfrak M(\mathbf{s,w})}$ (\cref{scatteringdiagramnewdef}).\end{definition}

Let $m\in M_\mathbb R$. It gives a King's stability condition on $J$-mod (\cref{semistablerep}). Let $\mathfrak M^{\text{ss}}(m)$ be the moduli of $m$-semistable $J$-modules (\cref{semistablerep}) which is an open substack of $\mathfrak M$. This gives us an element in the motivic Hall algebra
\[
1^{\mathrm{ss}}(m)\colon =[\mathfrak M^{\mathrm{ss}}(m)\hookrightarrow \mathfrak M]\in  \exp( \hat{{\mathfrak g}}^{\mathrm{Hall}}_{\mathbf{s,w}}).
\]

The following theorem is an incarnation of the existence and uniqueness of the Harder-Narasimhan filtration of objects in $J$-mod with respect to a stability condition.

\begin{theorem}[{\cite[theorem 6.5]{bridgeland2016scattering}}]\label{hallalgebrafactorization}
For the Hall algebra scattering diagram $\mathfrak D^{\mathrm{Hall}}_{\mathbf{s,w}}$ and any $m\in M_\mathbb R$, we have
\[
\phi_{\mathbf{s,w}}^\mathrm{Hall}(m) = 1^{\mathrm{ss}}(m)\in \exp( \hat{{\mathfrak g}}^{\mathrm{Hall}}_{\mathbf{s,w}}). 
\]
\end{theorem}

\begin{remark}
In \cite{bridgeland2016scattering}, the above theorem is proven for a generic $m$ on any wall. It can be extended to any $m\in M_\mathbb R$ in our definition of $\phi_\mathbf{s,w}^\mathrm{Hall}$ without any additional effort.
\end{remark}

\subsubsection{} We now review the so called integration map. It is the bridge connecting the Hall algebra scattering diagram and the stability scattering diagram to be defined later.

Define ${\mathfrak g}^{\mathrm{reg}}_{\mathbf{s,w}}$ to be the $\mathbb Q_{\mathrm{reg}}(q^\frac{1}{2})$-submodule of ${\mathfrak g}^{\mathrm{Hall}}_{\mathbf{s,w}}$ generated by the elements of the form $$(q^{1/2}-q^{-1/2})^{-1}[X\rightarrow \mathfrak M]$$ where $X$ is an algebraic variety. Define $H_\mathrm{reg}(\mathbf{s,w})$ to be the $\mathbb Q_{\mathrm{reg}}(q^\frac{1}{2})$-submodule of $H(\mathbf{s,w})$ generated by the elements of the form $[X\rightarrow \mathfrak M]$ where $X$ is an algebraic variety. Then there is a $\mathbb Q_\mathrm{reg}(q^\frac{1}{2})$-linear map
\[
\varphi \colon \mathfrak g^{\mathrm{reg}}_{\mathbf{s,w}}\rightarrow H_\mathrm{reg}(\mathbf{s,w}),\quad (q^{1/2}-q^{-1/2})^{-1}[X\rightarrow \mathfrak M]\mapsto [X\rightarrow \mathfrak M].
\]

\begin{theorem}[\cite{joyce2012theory}, \cite{bridgeland2016scattering}, {\cite[theorem 7.4]{nagao2013donaldson}}]\label{integration}
  We have the following properties regarding $\mathfrak g^{\mathrm{reg}}_{\mathbf{s,w}}$ and $H_\mathrm{reg}(\mathbf{s,w})$. 
  
  \begin{enumerate}[label=(\roman*)]
  	\item The submodule ${\mathfrak g}^{\mathrm{reg}}_{\mathbf{s,w}}$ is an $N^+_\mathbf{s}$-graded Lie subalgebra of $\mathfrak g_{\mathbf{s,w}}^\mathrm{Hall}$.
  	\item The submodule $H_\mathrm{reg}(\mathbf{s,w})$ is a subalgebra of $H(\mathbf{s,w})$. It is a Poisson algebra with the bracket
\[
\{a,b\} = (q^{1/2}-q^{-1/2})^{-1}[a,b]
\]  	
and the map $\varphi$ is a Lie algebra homomorphism.  	 
  	\item There is an $N^\oplus_\mathbf{s}$-graded Poisson homomorphism 
  \[
  I\colon H_\mathrm{reg}(\mathbf{s,w})\rightarrow  \mathbb Q[N^\oplus_\mathbf{s}],\quad I\left({[X\rightarrow \mathfrak M_d]}\right)= e(X)x^d
  \]
where $e(X)$ is the Euler characteristic of $X_\mathrm{an}$.
  \end{enumerate}
\end{theorem}

We thus also have a homomorphism of $N_\mathbf s ^+$-graded Lie algebras \begin{equation}\label{integrationmap}
\mathcal I= I\circ\varphi\colon \mathfrak g^{\mathrm{reg}}_{\mathbf{s,w}}\rightarrow \mathfrak g_\mathbf{s} = \mathbb Q[N_\mathbf{s}^+].	
\end{equation}

\subsubsection{}There is also a quantized version of the integration map $I$ when the potential $\mathbf w$ is \textit{polynomial}, i.e. $\mathbf w$ is in $\mathbb CQ$ but not a formal sum in $\widehat{\mathbb CQ}$. Let $\mathbb Q((q^\frac{1}{2}))$ be the field of formal Laurent series of the variable $q^{1/2}$. Recall that we can define a twisted product $*$ on $\mathbb Q((q^\frac{1}{2}))[N^\oplus_\mathbf{s}]$ as in (\ref{twistedprod}). For the definition of the \textit{quantum integration map} $I_q$ in the following theorem, see \cite[section 3]{davison2018positivity} and a review in \cite[section 2.6]{cheung2019donaldson}. 

\begin{theorem}[\cite{kontsevich2008stability}, {\cite[section 6.7]{davison2015donaldson}}]\label{qintegration}
Let $(\mathbf{s,w})$ be a seed with polynomial potential. There exists an $N_{\mathbf{s}}^\oplus$-graded algebra homomorphism 
 \[
 I_q\colon H(\mathbf{s,w})\rightarrow \left(\mathbb Q((q^\frac{1}{2}))[N^\oplus_\mathbf{s}],\,*\right).
 \]
 It induces a Poisson $\mathbb Q_\mathrm{reg}(q^\frac{1}{2})$-algebra homomorphism of subalgebras
 \[
 I_q\colon H_\mathrm{reg}(\mathbf{s,w})\rightarrow \left(\mathbb Q_\mathrm{reg}(q^\frac{1}{2})[N^\oplus_\mathbf{s}],\,*\right)
 \]
 whose semi-classical limit at $q^{1/2} = 1$ is exactly $I$ in \cref{integration}, i.e. 
 \[
	I = e_{q^{1/2}=1}\circ I_q. 
 \]
 \end{theorem}

It follows that there exists an induced Lie algebra homomorphism
\[
\mathcal I_q\colon \mathfrak g_\mathbf{s,w}^\mathrm{reg}\rightarrow \mathfrak g_\mathbf{s}^q,
\]
which can be viewed as a quantization of $\mathcal I$, i.e.
\[
\mathcal I = e_{q^{1/2} =1} \circ \mathcal I_q.
\]
To summarize, we have the following commutative diagram in the case that both \cref{integration} and \cref{qintegration} apply. Note that all the maps in the diagram extend to the completions.

\[
\begin{tikzcd}[row sep=large, column sep = large]
H_\mathrm{reg}(\mathbf{s,w})\arrow[r,"I_q",swap]\arrow[rr, "I", bend left=20]&\mathbb Q_{\mathrm{reg}}(q^\frac{1}{2})[N_\mathbf{s}^\oplus]\arrow[r,"q^{1/2}=1", swap] & \mathbb Q[N_\mathbf{s}^\oplus]\\
\mathfrak g_\mathbf{s,w}^\mathrm{reg}\ar[rr, "\mathcal I", bend right = 20, swap]\arrow[u,"\varphi"]\arrow[r, "\mathcal I_q"]&\mathfrak g^q_\mathbf{s} = \mathbb Q_{\mathrm{reg}}(q^\frac{1}{2})[\hat x^d, d\in N_\mathbf{s}^+]\arrow[u, "\hat x^d\mapsto x^d"]\ar[r, "q^{1/2}=1"]&\mathfrak g_\mathbf{s} = \mathbb Q[N_\mathbf{s}^+]\ar[u,hookrightarrow]
\end{tikzcd}
\]

\subsubsection{}The following \textit{absence of poles} theorem is due to Joyce; see also section 3.2 and definition 7.15 in \cite{joyce2012theory}.

\begin{theorem}[{\cite[theorem 8.7]{joyce2007configurations}}]\label{joyce}
Let $(\mathbf{s,w})$ be a seed with potential. For any $m \in M_\mathbb R$, we write $1^\mathrm {ss}(m) = 1+ \sigma$ for $\sigma\in \hat{\mathfrak g}^\mathrm{Hall}_{\mathbf{s,w}}$. Then we have that   
\[\log (1^{\mathrm{ss}}(m))\colon = \sigma -\frac{1}{2}\sigma *\sigma +\cdots + \frac{(-1)^{n-1}}{n}\sigma*\cdots *\sigma +\cdots
\]
belongs to the Lie subalgebra $\hat{\mathfrak g}^{\mathrm{reg}}_{\mathbf{s,w}}\subset \hat{\mathfrak g}^\mathrm{Hall}_{\mathbf{s,w}}$.
\end{theorem}

In the last section $\mathfrak D^{\mathrm{Hall}}_{\mathbf{s,w}}$ is defined as a ${\mathfrak g}^{\mathrm{Hall}}_{\mathbf{s,w}}$-SD. The above theorem in particular shows that $\mathfrak D^{\mathrm{Hall}}_{\mathbf{s,w}}$ is also a ${\mathfrak g}^{\mathrm{reg}}_{\mathbf{s,w}}$-SD. Recall that we have two $N^+_\mathbf{s}$-graded Lie algebra homomorphisms
\[
\mathcal I_q\colon {\mathfrak g}^{\mathrm{reg}}_{\mathbf{s,w}} \rightarrow \mathfrak g^q_{\mathbf{s}}\quad \mathrm{and}\quad \mathcal I\colon {\mathfrak g}^{\mathrm{reg}}_{\mathbf{s,w}} \rightarrow \mathfrak g_{\mathbf{s}}.
\]
By abuse of notation, we will denote the maps of corresponding pro-unipotent groups still by $\mathcal I_q$ and $\mathcal I$. Recall the conventions in \cref{mapsbetweensds}.

\begin{definition}[Stability scattering diagrams]\label{qstabilityscatteringdiagram} Let $(\mathbf {s,w})$ be a seed with potential. We define the \textit{stability scattering diagram} of $\mathbf {(s,w)}$ to be
 \[
 \mathfrak D_{\mathbf{s,w}} = (\mathfrak S_\mathbf{s,w}, \phi_\mathbf{s,w})\colon  = \mathcal I(\mathfrak D^{\mathrm{Hall}}_{\mathbf{s,w}}),
 \]
 i.e. the consistent $\mathfrak g_\mathbf{s}$-SD corresponding to  $\mathcal I(1_{\mathfrak{M}(\mathbf{s,w})})$.
If the potential $\mathbf w$ is polynomial, then we define the \textit{quantum stability scattering diagram} of $\mathbf {(s,w)}$ to be
 \[
 \mathfrak D^q_{\mathbf{s,w}} = (\mathfrak S_\mathbf{s,w}^q,\phi_\mathbf{s,w}^q)\colon  = \mathcal I_q(\mathfrak D^{\mathrm{Hall}}_{\mathbf{s,w}}),
 \]
 i.e. the consistent $\mathfrak g^q_{\mathbf{s}}$-SD corresponding to $\mathcal I_q(1_{\mathfrak{M}(\mathbf{s,w})})$. 
\end{definition}

\begin{remark}
The stability scattering diagram $\mathfrak D_{\mathbf{s,w}}$ is defined by Bridgeland in \cite[section 11]{bridgeland2016scattering} for a polynomial potential $\mathbf{w}$. However the definition can be easily extended to any potential using the integration map in \cref{integration}. The quantum integration map $I_q$ though, is only defined and proven to be an algebra homomorphism when the potential $\mathbf{w}$ is polynomial. Therefore whenever we talk about the quantum stability scattering diagram $\mathfrak D_{\mathbf{s,w}}^q$, we always assume $(\mathbf{s,w})$ to be a seed with polynomial potential.
\end{remark}

\begin{example}
Choosing a generic point $m$ in $[S_i]^\perp\subset M_\mathbb R$ for some vertex $i\in\{1,\dots, n\}$ of the quiver, the subcategory of $m$-semistable $J$-modules is generated by the simple module $S_i$. Then we have \[1^{\mathrm{ss}}(m) = \left[\coprod_{k\geq 0} BGL_k\rightarrow \coprod_{k\geq 0}\mathfrak M_{ks_i} \right]\] where the map is an isomorphism. This element is in $1 + \hat H(\mathbf{s,w})_{>0}$. By Joyce's absence of poles \cref{joyce}, we have
\[
I_q(1^{\mathrm{ss}}(m))\in \exp(\hat {\mathfrak g}_\mathbf{s,w}^\mathrm{reg}).
\]
One can show that (\cite[section 6.4]{kontsevich2008stability})
\[
I_q(BGL_k\rightarrow \mathfrak M_{ks_i}) = \frac{q^{k^2/2}x^{ks_i}}{[\mathrm{GL}_k]_q}
\]
and thus $I_q(1^{\mathrm{ss}}(m)) = \mathbb E(q^{1/2}, x^{s_i})$.
In terms of Lie algebra elements, we also have
\[
\mathcal I_q(\log(1^{\mathrm{ss}}(m))) = -\mathrm{Li}^q_2(-x^{s_i})
\quad\mathrm{and}\quad 
\mathcal I(\log(1^{\mathrm{ss}}(m)))= -\mathrm{Li}_2(-x^{s_i}).
\]
\end{example}

%% file: 4_mutation.tex
In this section, we first define two operations $\mu_{k}^+$ and $\mu_k^-$ on seeds with potentials (SPs for short) which are lifts of the mutations of quivers with potentials (QPs for short) in \cite{derksen2008quivers}. We then define functors relating the module category associated to an SP with the ones associated to the mutations of this SP. In the end, the relations between the Hall algebra scattering diagrams $\mathfrak D^{\mathrm{Hall}}_\mathbf{s,w}$ and $\mathfrak D^{\mathrm{Hall}}_{\mu_k^\pm(\mathbf{s,w})}$ (as well as the stability scattering diagrams) are studied with the help of these functors.

\subsection{Mutations of seeds with potentials}\label{mutationofsp}  

\subsubsection{}\label{mutationofquiver} We first review mutations of quivers.
Let $Q$ be a 2-acyclic quiver with the set of vertices $Q_0$ and the set of arrows $Q_1$. For an arrow $\alpha \in Q_1$, denote its source by $s(\alpha)$ and target by $t(\alpha)$. For a vertex $k\in Q_0$, we define the quiver $\tilde \mu_k(Q)$ obtained from $Q$ as follows:
\begin{enumerate}
\item[\textbf{1.}] For each pair of incoming arrow $\alpha\colon i\rightarrow k$ and outgoing arrow $\beta \colon k\rightarrow j$ at $k$, create one arrow $\overline{\beta \alpha}\colon i\rightarrow j$. 
\item[\textbf{2.}] Replace each incoming arrow $\alpha\colon i\rightarrow k$ with a new outgoing arrow $\bar \alpha\colon k\rightarrow i$; replace each outgoing arrow $\beta \colon k\rightarrow j$ with a new incoming arrow $\bar \beta \colon j\rightarrow k$.
\end{enumerate}
Now we perform the last step to get a 2-acyclic quiver $\mu_k(Q)$ from $\tilde\mu_k(Q)$.
\begin{enumerate}
    \item [\textbf{3.}] Delete a maximal collection of disjoint 2-cycles in $\tilde \mu_k(Q)$.
\end{enumerate}
It is easy to see the mutation $\mu_k$ is an involution, i.e. $\mu_k^2(Q)\cong Q.$
\subsubsection{}\label{mutationofqps} Before defining mutations of SPs, we review the mutations of QPs introduced in \cite{derksen2008quivers}. Let $(Q,\mathbf{w})$ and $(Q',\mathbf{w'})$ be two QPs with the same vertex set $Q_0$. We say that they are \textit{right-equivalent} (see \cite[definition 4.2]{derksen2008quivers}) if there is an isomorphism $\varphi\colon \widehat {\mathbb CQ} \rightarrow \widehat {\mathbb CQ'}$ preserving $Q_0$ and taking $\mathbf{w}$ to a potential $\varphi(\mathbf{w})$ cyclic equivalent to $\mathbf{w'}$. The                                                          isomorphism $\varphi$ also induces an isomorphism of the completed Jacobian algebras $J(Q,\mathbf{w})\cong J(Q',\mathbf{w}')$. Note that here we do not assume 2-acyclicity of the quivers.

A potential is called $\textit{reduced}$ if it contains no cycles of lengths less than 3; it is called $\textit{trivial}$ if the corresponding Jacobian algebra is trivial. For two QPs $(Q_1, \mathbf{w}_1)$ and $(Q_2, \mathbf{w}_2)$ with the identified set of vertices, their \textit{direct sum} is the QP $$(Q_1, \mathbf{w}_1)\oplus (Q_2, \mathbf{w}_2)\colon = (Q_1\oplus Q_2, \mathbf{w}_1+\mathbf{w}_2)$$ where $Q_1\oplus Q_2$ is the quiver on the same set of vertices but taking disjoint union of arrows of $Q_1$ and $Q_2$.

\begin{theorem}[{\cite[theorem 4.6]{derksen2008quivers}}]\label{splitting} Every QP $(Q,\mathbf{w})$  is right-equivalent to the direct sum of a trivial QP and a reduced QP
\[
(Q_\mathrm{triv}, \mathbf{w}_\mathrm{triv})\oplus (Q_\mathrm{red}, \mathbf{w}_\mathrm{red})
\]
where each direct summand is determined up to right-equivalence by the right equivalent class of $(Q,\mathbf{w})$. The Jacobian algebra of $(Q,\mathbf{w})$ is then isomorphic to the Jacobian algebra of the \textit {reduced part} $(Q_\mathrm{red}, \mathbf{w}_\mathrm{red})$ via the embedding of $(Q_\mathrm{red}, \mathbf{w}_\mathrm{red})$ in $(Q_\mathrm{triv}, \mathbf{w}_\mathrm{triv})\oplus (Q_\mathrm{red}, \mathbf{w}_\mathrm{red})$.
\end{theorem}

Now let $Q$ be a 2-acyclic quiver and $\mathbf{w}$ be a potential such that no term in its expansion starts at $k$ (if not, replace $\mathbf{w}$ with a cyclic equivalent one). We construct a potential for the quiver $\tilde \mu_k(Q)$ as follows. For each pair of incoming $\alpha$ and outgoing $\beta$ at vertex $k$, replace any occurrences of $\beta\alpha$ in $\mathbf{w}$ with $\overline{\beta\alpha}$. We thus get a potential $\overline {\mathbf{w}}$ in $\widehat{\mathbb C \tilde \mu_k(Q)}$. Define
$$
\tilde \mu_k(\mathbf{w})\colon =\overline {\mathbf{w}}+ \sum_{t(\alpha)= s(\beta) = k} \overline {\beta\alpha}\bar \alpha \bar \beta,
$$
where the sum is taken over all pairs of incoming $\alpha$ and outgoing $\beta$ at $k$. We thus have a QP $(\tilde \mu_k(Q), \tilde \mu_k(\mathbf{w})).$

\begin{definition}\label{mutationofqp}
The mutation $\mu_k(Q,\mathbf{w})$ of a 2-cyclic QP $(Q,\mathbf w)$ is defined to be the \textit{reduced part} of $(\tilde \mu_k(Q), \tilde \mu_k(\mathbf{w}))$ given by \cref{splitting}:
\[
\mu_k(Q,\mathbf{w})\colon = (\tilde \mu_k(Q)_\mathrm{red}, \tilde \mu_k(\mathbf{w})_\mathrm{red}).
\]
We say that a 2-acyclic QP $(Q, \mathbf{w})$ is \textit{mutable} at vertex $k$ (or $k$-\textit{mutable}) if the quiver $\tilde \mu_k(Q)_\mathrm{red}$ is equal to $\mu_k(Q)$ from \cref{mutationofquiver}. Therefore we have the justified notation $(\mu_k(Q),\mu_k(\mathbf{w}))$ of $\mu_k(Q,\mathbf{w})$ for $k$-mutable QPs. 
\end{definition}

The following proposition is an easy corollary of theorem 4.5 in \cite{derksen2008quivers}.
\begin{proposition}\label{mutationinvolution}
The mutation $\mu_k$ is an involution on the set of right-equivalent classes of $k$-mutable QPs, i.e. $\mu_k(Q,\mathbf{w})$ is also $k$-mutable and $\mu_k^2(Q,\mathbf{w})$ is right-equivalent to $(Q,\mathbf{w}).$
\end{proposition}

\subsubsection{} Denote the set of all seeds of $N$ by $\mathcal S$. Let $\mathbf{s}=\{s_1,\dots, s_n\}$ be a seed. As in \cref{motivichallalgebra}, the \textit{adjacency matrix} $B(\mathbf{s})$ is given by $B_{ij}=\{s_i,s_j\}$ which determines a quiver $Q(\mathbf{s})$. The two kinds of \textit{mutations} of seeds are the following maps from $\mathcal S$ to itself defined by

\[
\mu_k^+(\mathbf{s})_i\colon =
\begin{cases}
s_i+[-B_{ik}]_+s_k\ &\mathrm{for}\ i\neq k,\\
-s_k\ &\mathrm{for}\ i=k;
\end{cases}
\]

\[
\mu_k^-(\mathbf{s})_i\colon =
\begin{cases}
s_i+[B_{ik}]_+s_k\ &\mathrm{for}\ i\neq k,\\
-s_k\ &\mathrm{for}\ i=k.
\end{cases}
\]
Note that $\mu_k^+$ and $\mu_k^-$ are inverses to each other. That is
\begin{equation}\label{involutionseed}
\mu_k^+\circ \mu_k^-=\mu_k^-\circ \mu_k^+=\mathrm{id}.    
\end{equation}
An easy calculation shows $B(\mu_k^+(\mathbf s))=B(\mu_k^-(\mathbf s))$, i.e. two types of mutations of a seed give the same quiver:
\begin{equation}\label{compatible}
\mu_k(Q(\mathbf{s})) \cong Q(\mu_k^+(\mathbf{s})) \cong Q(\mu_k^-(\mathbf{s})).    
\end{equation}
Two SPs $(\mathbf{s,w})$ and $(\mathbf{s',w'})$ are said to be \textit{right-equivalent} if $\mathbf{s} = \mathbf{s'}$ and $(Q(\mathbf{s}),\mathbf{w})$ is right-equivalent to $(Q(\mathbf{s'}),\mathbf{w'})$ respecting the identification on vertices. An SP is said to be $k$-\textit{mutable} if the associated QP is $k$-mutable (\cref{mutationofqp}).

\begin{definition}[Mutation of SP]\label{spmutation} For a $k$-mutable seed with potential $(\mathbf{s,w})$, we define the following two \textit{mutations} at vertex $k$:
\[
\mu_k^+(\mathbf{s,w}) \colon =  (\mu_k^+(\mathbf{s}), \mu_k(\mathbf{w})),\quad \mu_k^-(\mathbf{s,w}) \colon =  (\mu_k^-(\mathbf{s}), \mu_k(\mathbf{w})).
\]
\end{definition}
\begin{proposition}\label{involutionsp}
The mutations of SPs $\mu_k^+$ and $\mu_k^-$ are inverse to each other on the set of right equivalent classes of $k$-mutable SPs.
\end{proposition}
\begin{proof}
As we point out in (\ref{compatible}), these mutations of SPs are compatible with mutations of QPs. Then by \cref{mutationinvolution}, both $\mu_k^+(\mathbf{s,w})$ and $\mu_k^-(\mathbf{s,w})$ are $k$-mutable and are inverse to each other by (\ref{involutionseed}).
\end{proof}

\subsection{Generalized reflection functors} 
\subsubsection{} Inspired by the mutations of decorated representations in \cite{derksen2008quivers}, for a $k$-mutable $(Q,\mathbf w)$, we define two functors 
\[
F_k^+,\ F_k^-\colon J(Q,\mathbf{w})\text{-mod}\rightarrow J(\mu_k(Q,\mathbf{w}))\text{-mod}.
\]
between module categories, generalizing the reflection functors of \cite{bernstein1973coxeter}.

Recall that there is an intermediate QP $\tilde \mu_k(Q,\mathbf{w})$ (\cref{mutationofquiver}) in the construction of $\mu_k(Q,\mathbf{w})$. There is an equivalence between module categories of the algebras $J(\tilde\mu_k (Q,\mathbf w))$ and $J(\mu_k (Q,\mathbf w))$ induced by the isomorphism of algebras in \cref{splitting}. The functors $F_k^\pm$ will be defined to be the following functors 
\[
\tilde F_k^\pm \colon J(Q,\mathbf w)\text{-mod} \rightarrow J(\tilde\mu_k (Q,\mathbf w))\text{-mod}
\]
post-composed by this equivalence. The functors $\tilde F_k^\pm$ are constructed as follows. Let $(Q,\mathbf{w})$ be a $k$-mutable QP and $M$ be an object in $J(Q,\mathbf{w})$-mod regarded as a nilpotent representation of the quiver $Q$ annihilated by $\overline{\langle\partial \mathbf{w}\rangle}$ (\cite[definition 10.1]{derksen2008quivers}). We denote by $Q_0$ the set of vertices of $Q$ and by $Q_1$ the set of arrows. We have two maps $s,t\colon Q_1\rightarrow Q_0$ which send an arrow to its source and target respectively.  We have the following diagram of vector spaces:
\[
\begin{tikzcd}
&&M_k\arrow[dr, "\beta_k"]&&\\
&M_{\textmd{in}}\arrow[ur, "\alpha_k"]&&M_{\textmd{out}}\arrow[dl,"q_k"]\arrow[ll, "\gamma_k"]&\\
&&\textmd{coker }\beta_k\arrow[ul,"\phi_k"]&&
\end{tikzcd}
\]
where 
$$M_{\textmd{in}}\colon =\bigoplus_{\alpha \in Q_1,\ t(\alpha)=k}M_{s(\alpha)},\   M_{\textmd{out}}\colon =\bigoplus_{\beta\in Q_1,\ s(\beta)=k}M_{t(\beta)}
$$
and
$$
\alpha_k=\sum_{t(\alpha)=k} M_{\alpha},\ \beta_k=\sum_{s(\beta)=k}M_{\beta},\ \gamma_k=\sum_{t(\alpha)=s( \beta)=k}\partial_{\beta\alpha}\mathbf{w}.$$ 
We have natural embeddings $\iota_\beta\colon M_{t(\beta)}\rightarrow M_\mathrm{out}$ and projections $\pi_\alpha\colon M_\mathrm{in}\rightarrow M_{s(\alpha)}$. One can check that $\alpha_k\circ \gamma_k=0$ and $\gamma_k\circ\beta_k=0$; see \cite[Lemma 10.6]{derksen2008quivers}. This implies $\gamma_k$ factors through $q_k$, i.e. $\gamma_k = \phi_k\circ q_k$ for a unique $\phi_k\colon \mathrm{coker} \beta_k\rightarrow M_\mathrm{in}$. We first define $M'=\tilde F^+_k(M)$ as a representation of $(Q',\mathbf{w'}) = \tilde \mu_k(Q,\mathbf{w})$ as follows.

\begin{enumerate}[label={(\roman*)}]
	\item $M'_k\colon =\textmd{coker } \beta_k$; $M'_i\colon =M_i$ for $i\neq k$.
	\item For each $\bar \beta\colon j\rightarrow k$, let $M'_{\bar \beta}$ be the composition $$M'_{\bar \beta}\colon M_j\xrightarrow {\iota_\beta} M_{\textmd{out}}\xrightarrow{q_k} \textmd{coker }\beta_k;$$ For each $\bar \alpha\colon k\rightarrow i$, let $M'_{\bar \alpha}$ be the composition $$M'_{\bar \alpha}\colon \textmd{coker }\beta_k\xrightarrow{\phi_k} M_{\textmd{in}}\xrightarrow{\pi_\alpha} M_i.$$
	\item For each $\overline{\beta\alpha}\colon i\rightarrow j$, let $M'_{\overline{\beta\alpha}}$ be $M_\beta\circ M_\alpha$. 
	\item For any $\gamma$ not incident to $k$, let $M'_\gamma = M_\gamma$.
\end{enumerate}

\begin{proposition}\label{thea}
	The above construction gives a representation $M'$ of $\tilde \mu_k(Q,\mathbf{w})$, i.e. $M'$ is a $J(\tilde \mu_k(Q,\mathbf{w}))$-module. Moreover, this defines an additive functor
	\[
	 \tilde F^+_k\colon J(Q,\mathbf{w})\text{-}\mathrm{mod}\rightarrow 
J(\tilde\mu_k(Q,\mathbf{w}))\text{-}\mathrm{mod}
	\]
	such that $\tilde F_k^+(M) = M'$.
\end{proposition}

\begin{proof} By construction, $M'$ is a finite dimensional representation of $Q'$. It is also nilpotent since $M$ is nilpotent. One needs to show that $M'$ is annihilated by $\partial_\gamma \mathbf{w'}$ for any arrow $\gamma$ of $Q'$. This is essentially checked in \cite[proposition 10.7]{derksen2008quivers} although the \textit{mutation of decorated representations} there is slightly different from ours at the vector space $M'_k$. To see how $\tilde F_k^+$ acts on the space of morphisms, suppose that $f\colon M\rightarrow V$ is a morphism in $J(Q,\mathbf{w})$-mod and we construct a morphism $\tilde F^+_k(f)\colon \tilde F^+_k(M)\rightarrow \tilde F^+_k(V)$ by giving maps between vector spaces associated to vertices of $Q'$. We keep the maps $f_i\colon M_i\rightarrow V_i$ unchanged if $i\neq k$ and construct a map from $M'_k$ to $V'_k$ by taking the naturally induced map between cokernels. Then by construction these maps intertwine with actions of arrows in $Q'$ and thus form a morphism $\tilde F_k^+(f)$ of representations. Other requirements of an additive functor should be easy to check. 	   
\end{proof}

To define $M^\circ=\tilde F^{-}_k(M) \in J(Q',\mathbf{w'})\textmd{-mod}$, we use the following diagram. 

\[
\begin{tikzcd}
&&M_k\arrow[dr, "\beta_k"]&&\\
&M_{\textmd{in}}\arrow[ur, "\alpha_k"]&&M_{\textmd{out}}\arrow[dl,"\psi_k"]\arrow[ll, "\gamma_k"]&\\
&&\textmd{ker }\alpha_k\arrow[ul,"r_k"]&&
\end{tikzcd}
\]
Note that $\gamma_k = r_k\circ \psi_k$. Define $M^\circ$ as follows.

\begin{enumerate}[label={(\roman*)}]
	\item $M^\circ_k\colon =\ker \alpha_k$; $M^\circ_i\colon =M_i$ for $i\neq k$.
	\item  For each $\bar \alpha\ colon k\rightarrow i$, let $M^\circ_{\bar\alpha}$ be the composition
	$$M^\circ_{\bar\alpha}\colon \ker \alpha_k \xrightarrow {r_k}M_{\textmd{in}}\xrightarrow{\pi_\alpha} M_i;$$
	For each $\bar \beta\colon j\rightarrow k$ in $Q'$, let $M^\circ_{\bar \beta}$ be the composition $$M_{\bar \beta}^\circ\colon M_j\xrightarrow{\iota_\beta} M_{\textmd{out}}\xrightarrow {\psi_k}\ker \alpha_k.$$
	\item For each $\overline{\beta\alpha}\colon i\rightarrow j$, let $M^\circ_{\overline{\beta\alpha}}$ be $M_\beta\circ M_\alpha$.
	\item For any $\gamma$ not incident to $k$, let $M^\circ_\gamma = M_\gamma$.
\end{enumerate}

\begin{proposition}\label{tacgar}
The above construction gives a representation $M^\circ$ of $\tilde \mu_k(Q,\mathbf{w})$, i.e. $M^\circ$ is a $J(\tilde \mu_k(Q,\mathbf{w}))$-module. Moreover, this defines an additive functor
	\[
	 \tilde F^-_k\colon J(Q,\mathbf{w})\text{-}\mathrm{mod}\rightarrow 
J(\tilde\mu_k(Q,\mathbf{w}))\text{-}\mathrm{mod}
	\]
	such that $\tilde F_k^-(M) = M^\circ$.
\end{proposition}
\begin{proof}
	The proof is analogous to \cref{thea}. Similarly, this defines a functor since the map between kernels is naturally induced.
\end{proof}

Now we have two functors $\tilde F^{\pm}_k\colon J(Q,\mathbf{w})\text{-}\mathrm{mod}\rightarrow 
J(\tilde\mu_k(Q,\mathbf{w}))\text{-}\mathrm{mod}$. To extend the target of these functors to $J(\mu_k(Q,\mathbf{w}))\text{-}\mathrm{mod}$, one uses the following equivalence
\[
R\colon J(\tilde\mu_k(Q,\mathbf{w}))\text{-}\mathrm{mod} \rightarrow J(\mu_k(Q,\mathbf{w}))\text{-}\mathrm{mod}
\]
described in \cref{splitting}.

\begin{definition}[Generalized reflection functors]\label{reflectionqp}
For a $k$-mutable QP $(Q,\mathbf{w})$, we define two functors between module categories:
\[
F_k^\pm\colon = R\circ \tilde F_k^\pm \colon J(Q, \mathbf {w})\text{-mod} \rightarrow J(\mu_k(Q, \mathbf{w}))\text{-mod}.
\]
These are what we called the \textit{generalized reflection functors} or \textit{mutations of representations}.
\end{definition}

\subsubsection{} In what follows, we work with SPs instead of QPs. Note that two mutations $\mu_k^\pm$ of SPs give the same QP: 
\[
\mu_k(Q(\mathbf{s}),\mathbf{w}) \cong (Q(\mu_k^\pm (\mathbf{s})), \mu_k(\mathbf{w})).
\]
We assign one mutation of representations to one mutation of SPs of the corresponding sign. That is we define
\begin{equation}\label{reflectionsp}
F_k^\pm\colon = R\circ \tilde \mu_k^\pm \colon J(\mathbf {s,w})\text{-mod} \rightarrow J(\mu_k^\pm (\mathbf{s,w}))\text{-mod}.
\end{equation}
The advantage of working with SPs instead of QPs, which we shall explain below, is that it is easier to see how dimension vectors of certain representations get reflected under mutations. For any SP $(\mathbf{s,w})$, the Grothendieck group $K_0(J(\mathbf{s,w})\text{-mod})$ is naturally identified with the lattice $N$ via $[S_i]=s_i$. Let $(\mathbf{s',w'}) = \mu_k^+ (\mathbf{s,w})$ and $(\mathbf{s'',w''}) = \mu_k^- (\mathbf{s,w})$. Straightforward calculations show 
\begin{equation}\label{dimensionvectors}
[S_i]=
\begin{cases}
[F_k^{+}(S_i)]=[F_k^-(S_i)]\ &\mathrm{for}\ i\neq k,\\
-[S'_k]=-[S''_k]\ &\mathrm{for}\ i=k.
\end{cases}
\end{equation}
This observation can be generalized. There are actually subcategories whose objects' dimension vectors in $N$ are invariant under mutations. Denote $J(\mathbf{s,w})\text{-mod}$ and $J(\mathbf{s',w'})\text{-mod}$ by $\mathcal A$, $\mathcal A'$ respectively. We define the following full subcategories of $\mathcal A$ (and of $\mathcal A'$ accordingly)
\[
 \mathcal A_{k,-}= {^{\perp}S_{k}} \colon = \{ M \in \mathcal A\mid \mathrm{Hom}   (M,S_k) =0 \};
\]
\[
\mathcal A_{k,+} = S_k^\perp\colon =\{M\in \mathcal A\mid \mathrm {Hom}(S_k,M)=0\}
\]
and denote by $\langle S_k \rangle$ the full subcategory of $\mathcal A$ consisting of direct sums of $S_k$.

Note that $\mu_k^-(\mathbf{s',w'})$ is right-equivalent to $(\mathbf{s,w})$ by \cref{involutionsp}. So the functor $F_k^-$ can be regarded as from $\mathcal A'$ to $\mathcal A$. 
\begin{theorem}\label{equiv}
The generalized reflection functors $F_k^+\colon \mathcal A\rightarrow \mathcal A'$ and $F_k^- \colon \mathcal A' \rightarrow \mathcal A$ have the following properties.
\begin{enumerate}[label={(\roman*)}]
    \item $F_k^+$ is right exact and $F_k^-$ is left exact. They are adjoint to each other.
    \item $F_k^+(\langle S_k\rangle) = 0$ and $F_k^-(\langle S'_k\rangle) = 0$.
    \item $F_k^+(\mathcal A_{k,+})\subset A'_{k,-}$ and $F_k^-(\mathcal A'_{k,-})\subset A_{k,+}$.
    \item The restrictions $F_k^+\colon \mathcal A_{k,+}\rightarrow \mathcal A_{k,-}^{'}$ and $F_k^- \colon \mathcal A_{k,-}^{'} \rightarrow \mathcal A_{k,+}$ are inverse equivalences. Moreover, these equivalences preserve short exact sequences.
    \item If $V$ is in $\mathcal A_{k,+}$, then $[V] = [F_k^+(V)]\in N$ and if $W$ is in $\mathcal A'_{k,-}$, then $[W] = [F_k^-(W)]\in N$.
\end{enumerate}
\end{theorem}

\begin{proof} For the first part in $(i)$, it suffices to prove $\tilde F_k^+$ is right exact and $\tilde F_k^-$ is left exact. Suppose we have an exact sequence 
\[
U\rightarrow V\rightarrow W\rightarrow 0.
\]
in $\mathcal A$. This implies we have an exact sequence of complexes of vector spaces concentrated on degree $1$ and $0$ which induces an exact sequence on $H_0$'s (the first row of the following diagram).
\[
\begin{tikzcd}
\tilde F_k^+(U)_k\arrow[r]&\tilde F_k^+(V)_k\arrow[r]&\tilde F_k^+(W)_k\arrow[r]&0\\
U_\mathrm{out}\arrow[r]\arrow[u]&V_\mathrm{out}\arrow[r]\arrow[u]&W_\mathrm{out}\arrow[r]\arrow[u]&0\\
U_{k}\arrow[r]\arrow[u,"\beta_k"]&V_{k}\arrow[r]\arrow[u,"\beta'_k"]&W_{k}\arrow[r]\arrow[u,"\beta''_k"]&0\\
\end{tikzcd}
\]
This shows $\tilde F_k^+$ is right exact and so is $F_k^+$. The left exactness of $F_k^-$ is proven similarly. 

For adjointness, let $V$ be in $\mathcal A$ and $W$ be in $\mathcal A'$. We need to show there is a natural isomorphism between
\[
\mathrm{Hom}_\mathcal A(V, F_k^-(W)) \cong \mathrm{Hom}_{\mathcal A'}(F_k^+(V), W). 
\]
$\mathrm{Hom}_\mathcal A(V, F_k^-(W))$ is given by the space of linear maps $\{f_i\colon V_i\rightarrow F_k^-(W)_i\}_{i=1}^n$ intertwining with the action of $J(\mathbf{s,w})$. However, since in representation $F_k^-(W)$, the map $F_k^-(W)_k\rightarrow W_\mathrm{in}$ is injective, $f_k$ is uniquely determined by $f\colon V_\mathrm{out}\rightarrow W_\mathrm{in}$ to let the following diagram commute 
\[
\begin{tikzcd}
V_k\arrow[r,"\beta^V_k"]\arrow[d, "f_k"]&V_\mathrm{out}\arrow[d,"f"]\\
F_k^-(W)_k\arrow[r,"r_k"]&W_\mathrm{in}.
\end{tikzcd}
\]
The existence of $f_k$ is then given by the condition that the following composition
\[
\begin{tikzcd}
V_k\arrow[r,"\beta^V_k"]&V_\mathrm{out}\arrow[r,"f"]&W_\mathrm{in}\arrow[r,"\alpha^W_k"]&\mathrm{coker}(r_k) = W_k
\end{tikzcd}
\]
is zero. Then $\mathrm{Hom}_\mathcal A(V,F_k^-(W)$ is the space of maps $\{f_i\}_{i\neq k}$ intertwining with the action of $J(\mathbf{s,w})$ such that $\alpha^W_k\circ f\circ \beta_k^V$ is zero. Similarly $\mathrm{Hom}_{\mathcal A'}(F_k^+(V),W)$ is the space of maps $\{g_i\colon F_k^+(V)_i\rightarrow W_i \}_{i\neq k}$ intertwining with the action of $J(\mu_k^+(\mathbf{s,w}))$ such that $\alpha^W_k\circ g\circ \beta_k^V$ is zero.

Note that $V_i$ is canonically identified with $F_k^+(V)_i$ if $i\neq k$ and so is for $W$ and $F_k^-(W)$. It is proved in \cite[corollary 6.6]{derksen2008quivers} that the subalgebras are isomorphic 
\[
\bigoplus_{i,j\neq k}e_iJ(\mathbf{s,w})e_j\cong
\bigoplus_{i,j\neq k}e_iJ(\mu_k^+(\mathbf{s,w}))e_j
\]
and by the construction of $F_k^\pm$ on representations, the actions of these subalgebras on $\bigoplus_{i\neq k} V_i$ and $\bigoplus_{i\neq k}F_k^+(V)_i$ are also identified via the isomorphism. The same is true for $W$ and $F_k^-(W)$. It then follows that $\mathrm{Hom}_{\mathcal A'}(F_k^+(V), W)$ is given by the exact same space as $\mathrm{Hom}_\mathcal A(V,F_k^-(W))$. This proves the adjointness of $F_k^+$ and $F_k^-$.

$(ii)$ and $(iii)$ follow directly from the construction of $F^{\pm}_k$.

By the adjointness and $(iii)$, to prove the first part of $(iv)$, we only need to show 
\[
F_k^-(F_k^+(V))\cong V\quad \mathrm{and}\quad
F_k^+(F_k^-(W))\cong V\quad
\]
if $V\in \mathcal A_{k,+}$ and $W\in \mathcal A'_{k,-}$. Already in the argument of proving the adjointness, we know the action of the subalgebra $\bigoplus_{i,j\neq k}e_iJ(\mathbf{s,w})e_j$ does not change on $F_k^-(F_k^+(V))$. Then one can easily check the actions of arrows adjacent to vertex $k$ also do not change. To prove that the exactness is preserved, we use the same diagram as in the proof of $(i)$. Suppose now we have a short exact sequence of objects in $\mathcal A_{k,+}$
\[
0\rightarrow U\rightarrow V \rightarrow W\rightarrow 0.
\]
Note that the exact sequence of complexes induces a long exact sequence also involving $H_1$'s:
\[
\begin{tikzcd}
\dots\arrow[r]&\ker \beta_k''\arrow[r]&\tilde F_k^+(U)_k\arrow[r]&\tilde F_k^+(V)_k\arrow[r]&\tilde F_k^+(W)_k\arrow[r]&0.
\end{tikzcd}
\]
However $\ker \beta_k''$ vanishes by our assumption $W\in \mathcal A_{k,+}$. Then the exactness follows. The proof for $F_k^-$ is similar.

$(v)$ is a direct computation as in (\ref{dimensionvectors}).
\end{proof}

\begin{remark}\label{ses}
The subcategory $\mathcal A_{k,-}$ is closed under taking quotients while the subcategory $\mathcal A_{k,+}$ is closed under taking subobjects.
\end{remark}

\subsection{Semistable representations} As mentioned in last section, the Grothendieck group $K_0(\mathcal A)$ is identified with the lattice $N$ in a natural way where $\mathcal A$ denotes the module category $J(\mathbf{s,w})$-mod. For $m\in M_\mathbb R$, denote the natural pairing $m([V])$ by $m(V)$.

\begin{definition}\label{semistablerep}
Given $m\in M_\mathbb R$, a module $V\in \mathcal A$ is $m$\textit{-semistable} (resp.\textit{-stable}) if
\begin{enumerate}[label={(\roman*)}]
    \item $m(V)=0$;
    \item $m(W)\leq 0$ (resp. $<0$) for any non-zero proper submodule $W\subset V$.
\end{enumerate}
\end{definition}

\begin{lemma}
Let $m\in M_\mathbb R$ such that $m(s_k)<0$ and $V$ be a module in the subcategory $\mathcal A_{k,-}\subset \mathcal A$. Then $V$ is $m$-semistable if and only if 
\begin{enumerate}[label={(\roman*)}]
    \item $m(V)=0$;
    \item for any $W\in \mathcal A_{k,-}$ and $W\subset V$, $m(W)\leq 0$.
\end{enumerate}
\end{lemma}

\begin{proof}
The only if part follows from definition. We prove the if part. Let $V$ be a $J$-module in $\mathcal A_{k,-}$. Then every submodule $W$ of $V$ has a unique maximal submodule $W'$ without any quotient isomorphic to $S_k$, i.e. $W'\in \mathcal A_{k,-}$. Since $m(s_k)<0$, we have $m(W')\geq m(W)$. Therefore in order to check the semistability of $V$, it suffices to examine all its subobjects in $\mathcal A_{k,-}$.
\end{proof}

We also have the following analogous lemma for $\mathcal A_{k,+}$. The proof is similar.

\begin{lemma}
Let $m\in M_\mathbb R$ such that $m(s_k)>0$ and $V$ be a module in the subcategory $\mathcal A_{k,+}\subset \mathcal A$. Then $V$ is $m$-semistable if and only if 
\begin{enumerate}[label={(\roman*)}]
    \item $m(V)=0$;
    \item for any $W\in \mathcal A_{k,+}$ being a quotient object of $V$, $m(W)\geq 0$.
\end{enumerate}
\end{lemma}

By \cref{ses}, one can rephrase above two lemmas in terms of short exact sequences in $\mathcal A_{k,+}$ and $\mathcal A_{k,-}$ as follows.

\begin{lemma}\label{rephrase}
Let $m\in M_\mathbb R$ such that $m(s_k)>0$ (resp. $<0$). Let $V$ be a module in the subcategory $\mathcal A_{k,+}\subset \mathcal A$ (resp. $\mathcal A_{k,-}$). Then $V$ is $m$-semistable if and only if 
\begin{enumerate}[label={(\roman*)}]
    \item $m(V)=0$;
    \item  for any short exact sequence 
$
0\rightarrow V'\rightarrow V\rightarrow V''\rightarrow0
$
in $\mathcal A_{k,+}$ (reps. $\mathcal A_{k,-}$), $m(V')\leq 0$.
\end{enumerate}
\end{lemma}

Recall that we put $(\mathbf{s''},\mathbf{w''}) = \mu_k^-(\mathbf{s,w})$ and denote $J(\mathbf{s''},\mathbf{w''})$ by $\mathcal A''$. The following is the key proposition that connects semistable representations in $\mathcal A$ with the ones in $\mathcal A'$ and $\mathcal A''$. 

\begin{proposition}\label{ifm}
Let $V$ be in $\mathcal A$ not isomorphic to $S_k$ and $m\in M_\mathbb R$. If $m(s_k)>0$ (resp. $<0$), then $V$ is $m$-semistable if and only if $F_k^+(V)\in \mathcal A'$ (resp. $F_k^-(V)\in \mathcal A''$) is $m$-semistable. 
\end{proposition}

\begin{proof}

First note that we only need to prove the case $m(s_k)>0$ and the other case follows. Now let $V$ be any $m$-semistable object in $\mathcal A$ and $m(s_k)>0$. Immediately, $V\in \mathcal A_{k,+}$ since $V$ has no $S_k$ as subobject. By \cref{rephrase}, $V$ is $m$-semistable if and only if for every short exact sequence $0\rightarrow V'\rightarrow V\rightarrow V''\rightarrow 0$ in $\mathcal A_{k,+}$, $m(V')< 0$. Apply the functor $F_k^+\colon \mathcal A_{k,+}\rightarrow \mathcal A_{k,-}^{'}$ which preserves short exact sequences and also does not change their dimension vectors in $N$ by \cref{equiv}. Therefore $V$ is $m$-semistable if and only if for every short exact sequence $0\rightarrow W'\rightarrow \mu_k^+(V)\rightarrow W''\rightarrow 0$ in $\mathcal A_{k,+}$, $m(W')\leq 0$, which is equivalent to $F_k^+(V)$ being $m$-semistable in $\mathcal A'$ again by \cref{rephrase}.

\end{proof}

Denote the subcategory of $m$-semistable modules in $\mathcal A$ by $\mathcal A(m)$. Summarizing the results in this section and by (iv) of \cref{equiv}, we have the following proposition. 
\begin{proposition}\label{fmg0}
For any $m\in M_\mathbb R$ such that $m(s_k)>0$ (resp. $<0$), the functor $F_k^+$ (resp. $F_k^-$) $\colon \mathcal A(m) \rightarrow \mathcal A'(m)$ (resp. $\mathcal A''(m)$) is an equivalence of abelian categories.
\end{proposition}

\subsection{Scattering diagrams under mutations} 

\subsubsection{} Let $(\mathbf{s,w})$ be a $k$-mutable SP for some vertex $k$ and $(\mathbf{s',w'})=\mu_k^+(\mathbf{s,w})$. Now we study the relation between $\mathfrak D_{\mathbf{s,w}}^\mathrm{Hall}$ and $\mathfrak D_{\mathbf{s',w'}}^\mathrm{Hall}$. 

As in last section, let $\mathcal A = J(\mathbf{s,w})$-mod and denote the motivic Hall algebra $H(\mathbf{s,w})$ by $H(\mathcal A)$. The subcategory $\mathcal A_{k,+}\subset \mathcal A$ is closed under extension. Then inside the algebra $H(\mathcal{A})$, we have a subalgebra $H(\mathcal A_{k,+})$ as the Hall algebra of $\mathcal A_{k,+}$. The equivalence $F_k^+\colon \mathcal A_{k,+}\rightarrow \mathcal A'_{k,-}$ induces an isomorphism of Hall algebras (which extends to completions)
\[
f_k^+\colon H(\mathcal A_{k,+})\longrightarrow H(\mathcal A'_{k,-}). 
\]
We are going to compare the induced functions
\[
\phi_\mathbf{s,w}^\mathrm{Hall}\colon M_\mathbb R\rightarrow \exp(\hat {\mathfrak g}^{\mathrm{Hall}}_{\mathbf{s,w}})\subset \hat H(\mathcal A)\quad 
\mathrm{and}\quad \phi^\mathrm{Hall}_\mathbf{s',w'}\colon M_\mathbb R\rightarrow \exp(\hat {\mathfrak g}^{\mathrm{Hall}}_{\mathbf{s',w'}})\subset \hat H(\mathcal A').
\]
Note that the vector space $M_\mathbb R$ is divided by $s_k$ into two open half spaces
 \[
 \mathcal H_{\mathbf{s}}^{k,+}=\{m\ |\ m(s_k)> 0\},\   \mathcal H_{\mathbf{s}}^{k,-}=\{m\ |\ m(s_k)< 0\}.
 \]
It is clear that $\mathcal H_{\mathbf{s}}^{k,\bullet}$ is constructible with respect to the decomposition $\mathfrak S^{\mathrm{Hall}}_\mathbf{s,w}$, and also to $\mathfrak S^{\mathrm{Hall}}_\mathbf{s',w'}$ since a cone in any of these two decompositions never crosses $s_k^\perp$.
\begin{theorem}\label{hallalgebramutation}
For any $m\in \mathcal H_\mathbf{s}^{k,+}$, we have that $\phi^\mathrm{Hall}_\mathbf{s,w}(m)$ is in $H(\mathcal A_{k,+})$ and $\phi^\mathrm{Hall}_\mathbf{s',w'}(m)$ is in $H(\mathcal A'_{k,-})$ and
\[
f_k^+(\phi^\mathrm{Hall}_\mathbf{s,w}(m)) = \phi^\mathrm{Hall}_\mathbf{s',w'}(m)
\]
where $f_k^+\colon H(\mathcal A_{k,+})\rightarrow H(\mathcal A'_{k,-})$ is the algebra isomorphism induced by the equivalence $F_k^+\colon \mathcal A_{k,+}\rightarrow \mathcal A_{k,-}'$. We thus say that the scattering diagrams $\mathfrak D_{\mathbf{s,w}}$ and $\mathfrak D_{\mathbf{s',w'}}$ are \textit{equal} on the half space $\mathcal H^{k,+}_{\mathbf{s}}.$
\end{theorem}

\begin{proof}
By \cref{ifm} and \cref{fmg0}, the equivalence $F_k^+\colon \mathcal A_{k,+}\rightarrow \mathcal A'_{k,-}$ restricts to equivalence $\mathcal A(m) \cong \mathcal A'(m)$ where $m$ is any point in $\mathcal H_\mathbf{s}^{k,+}$. These two subcategories represent $1_\mathbf{s,w}^\mathrm{ss}(m)\in H(\mathcal A_{k,+})$ and $1_\mathbf{s',w'}^\mathrm{ss}(m)\in H(\mathcal A'_{k,-})$ respectively. Thus by \ref{hallalgebrafactorization}, we have
\[
f_k^+(\phi^\mathrm{Hall}_\mathbf{s,w}(m)) = f_k^+(1_\mathbf{s,w}^\mathrm{ss}(m)) = 1_\mathbf{s',w'}^\mathrm{ss}(m) = \phi^\mathrm{Hall}_\mathbf{s',w'}(m).
\]
Note that by \cref{pathconnectedcomponent}, the cone decomposition $\mathfrak S_{\mathbf {s,w}}^\mathrm{Hall}$ is determined by the function $\phi_\mathbf{s,w}^\mathrm{Hall}$. Thus $\mathfrak S_{\mathbf {s,w}}^\mathrm{Hall}$ and $\mathfrak S_{\mathbf {s',w'}}^\mathrm{Hall}$ restrict to the same cone decomposition on $\mathcal H_\mathbf {s}^{k,+}$.
\end{proof}
Recall that we have another mutated SP $\mu_k^-(\mathbf{s,w})$. Since $F_k^+$ and $F_k^-$ are inverse to each other, the following corollary is immediate.

\begin{corollary}\label{hallalgebramutation-}
In the sense of \cref{hallalgebramutation}, the scattering diagrams ${\mathfrak D}_\mathbf{s,w}$ and ${\mathfrak D}_{\mu_k^-(\mathbf{s,w})}$ are equal on the half space $\mathcal H_{\mathbf{s}}^{k,-}$.
\end{corollary}

\subsubsection{} 

Let us now consider the stability scattering diagrams $\mathfrak D_{\mathbf{s,w}}$ and $\mathfrak D_{\mathbf{s',w'}}$. They have the associated functions
\[
\phi_\mathbf{s,w}\colon M_\mathbb R\rightarrow \exp(\hat {\mathfrak g}_{\mathbf{s}})\quad 
\mathrm{and}\quad \phi_\mathbf{s',w'}\colon M_\mathbb R\rightarrow \exp(\hat {\mathfrak g}_{\mathbf{s'}}).
\]
Denote $N_\mathbf{s}^+\cap N_\mathbf{s'}^+$ by $N^+_{\mathbf{s}\cap \mathbf{s'}}$ and define
\[
\mathfrak g_{\mathbf{s}\cap\mathbf{s'}}\colon = \mathbb Q[x^d, d\in N^+_{\mathbf{s}\cap\mathbf{s'}}].
\]
This is a common Lie subalgebra of ${\mathfrak g}_{\mathbf{s}}$ and ${\mathfrak g}_{\mathbf{s'}}$. As in the case of the Hall algebra scattering diagrams, the two half spaces $\mathcal H_{\mathbf{s}}^{k,\pm}$ are also constructible with respect to $\mathfrak S_\mathbf{s,w}$ and to $\mathfrak S_\mathbf{s',w'}$.

\begin{theorem}\label{stabilitymutation}
For any $m\in \mathcal H_\mathbf{s}^{k,+}$, we have $\phi_\mathbf{s,w}(m) = \phi_\mathbf{s',w'}(m)\in \exp(\hat {\mathfrak g}_{\mathbf{s}\cap \mathbf{s'}})$. We thus say $\mathfrak D_{\mathbf{s,w}}$ and $\mathfrak D_{\mathbf{s',w'}}$ are \textit{equal} on $\mathcal H^{k,+}_{\mathbf{s}}.$\end{theorem}

\begin{proof}
This immediately follows from \cref{hallalgebramutation} and the integration map $\mathcal I$ in \cref{integration}. In fact, 
\[
\phi_\mathbf{s,w}(m) = \exp(\mathcal I(\log (1_{\mathbf{s,w}}^\mathrm{ss}(m)))) = \exp(\mathcal I(\log (1_{\mathbf{s',w'}}^\mathrm{ss}(m)))) = \phi_\mathbf{s',w'}(m).
\]
\end{proof}

\begin{corollary}\label{stabilitymutation-}
In the sense of \cref{stabilitymutation}, the scattering diagrams ${\mathfrak D}_\mathbf{s,w}$ and ${\mathfrak D}_{\mu_k^-(\mathbf{s,w})}$ are equal  on $\mathcal H_{\mathbf{s}}^{k,-}$
\end{corollary}

The above discussion can be extended to the quantum stability scattering diagrams $\mathfrak D^q_{\mathbf{s,w}}$ and $\mathfrak D^q_{\mathbf{s',w'}}$ (assuming polynomial potentials). They have the associated functions
\[
\phi_\mathbf{s,w}^q\colon M_\mathbb R\rightarrow \exp(\hat {\mathfrak g}^q_{\mathbf{s}})\quad 
\mathrm{and}\quad \phi^q_\mathbf{s',w'}\colon M_\mathbb R\rightarrow \exp(\hat {\mathfrak g}^q_{\mathbf{s'}}).
\]
We define
\[
\mathfrak g_{\mathbf{s}\cap\mathbf{s'}}^q\colon = \mathbb Q_\mathrm{reg}(q^{1/2})[\hat x^d, d\in N^+_{\mathbf{s}\cap\mathbf{s'}}].
\]
This is a common Lie subalgebra of ${\mathfrak g}^q_{\mathbf{s}}$ and ${\mathfrak g}^q_{\mathbf{s'}}$. Thanks to the existence of the quantum integration map (\cref{qintegration}), immediately we have the following:

\begin{theorem}\label{qstabilitymutation}
For any $m\in \mathcal H_\mathbf{s}^{k,+}$, we have $\phi^q_\mathbf{s,w}(m) = \phi^q_\mathbf{s',w'}(m)\in \exp(\hat {\mathfrak g}^q_{\mathbf{s}\cap \mathbf{s'}})$, i.e. $\mathfrak D^q_{\mathbf{s,w}}$ and $\mathfrak D^q_{\mathbf{s',w'}}$ are \textit{equal} on $\mathcal H^{k,+}_{\mathbf{s}}.$ Similarly, ${\mathfrak D}^q_\mathbf{s,w}$ and ${\mathfrak D}^q_{\mu_k^-(\mathbf{s,w})}$ are equal  on $\mathcal H_{\mathbf{s}}^{k,-}$.
\end{theorem}

\subsubsection{} Recall we let $(\mathbf{s',w'}) =\mu_k^+(\mathbf{s,w})$ and $(\mathbf{s'',w''}) =\mu_k^-(\mathbf{s,w})$. Combining \cref{stabilitymutation} and \cref{stabilitymutation-}, one can view that $\mathfrak D_{\mathbf{s,w}}$ is glued together by the negative half of $\mathfrak D_{\mu_k^+(\mathbf{s,w})}$ and the positive half of $\mathfrak D_{\mu_k^-(\mathbf{s,w})}$ along the hyperplane $s_k^\perp$. The QPs associated to $\mu_k^+(\mathbf{s,w})$ and $\mu_k^-(\mathbf{s,w})$ are identified and thus have equivalent module categories. The natural identifications of $K_0(\mathcal A')$ and $K_0(\mathcal A'')$ with $N$ differ by an automorphism $T_k^\vee$ of $N$:
\[
 T_k^\vee(s''_i)=s'_i,\ \mathrm{for}\ i\in \{1,\dots, n\}.
\]
We call the dual linear transformation $ T_k\colon M_\mathbb R\rightarrow M_\mathbb R$. Concretely, we have
\[
T_k(m) = m + p^*(s_k)\langle s_k, m \rangle
\]
and in particular
\[
 T_k(f'_i)=f''_i,\ \mathrm{for}\ i\in \{1,\dots, n\}
\]
where $\mathbf{f} = \{f_1,\dots, f_n\}$ is the dual basis of $\mathbf{s}$ ($\mathbf{f'}$ and $\mathbf{f''}$ similarly defined). The linear transformation $T_k$ sends cones in the cone decomposition $\mathfrak S_{\mathbf{s',w'}}$ to the ones in $\mathfrak S_{\mathbf{s'',w''}}$, i.e.
\[
\mathfrak S_{\mathbf{s'',w''}} = T_k (\mathfrak S_{\mathbf{s',w'}})
\]
and also we have the following relation on associated functions
\[
\phi_{\mathbf{s'',w''}} = (T_k^\vee)^{-1}\circ \phi_{\mathbf{s',w'}}\circ T_k^{-1}
\]
where $(T_k^\vee)^{-1}$ denotes the induced isomorphism from $\exp(\hat T_{\mathbf{s',w'}})$ to $\exp(\hat T_{\mathbf{s'',w''}})$.

Note that the linear transformation $T_k$ fixes the hyperplane $s_k^\perp$ and preserves $\mathcal H^{k,+}_\mathbf{s}$ and $\mathcal H^{k,-}_\mathbf{s}$ respectively. We define the following piecewise linear transformation $T_k^-$ of $M_\mathbb R$
\[
T_k^-(m)\colon = 
\begin{cases}
 T_k(m) = m + p^*(s_k)\langle s_k, m \rangle\ &\mathrm{if}\ m\in \mathcal H^{k,+}_\mathbf{s},\\
m\ &\mathrm{if}\ m\in \mathcal H^{k,-}_\mathbf{s}.
\end{cases}
\]

To summarize, as a rephrasing of \cref{stabilitymutation} and \cref{stabilitymutation-}, we have the following description of mutations for stability scattering diagrams. This theorem should be compared with mutations of cluster scattering diagrams described in \cite[theorem 1.24]{gross2018canonical} where their mutation of seed $\mu_k$ is our $\mu_k^-$ and their piecewise linear map $T_k$ is our $T_k^-$. Note that the above discussion also applies with little change to Hall algebra and quantum stability scattering diagrams. 

\begin{theorem}[mutation of stability scattering diagrams]\label{scatmutation}
Let $(\mathbf{s,w})$ be a $k$-mutable SP. Restricted on the subset $\mathcal H_\mathbf{s}^{k,-}\sqcup \mathcal H_\mathbf{s}^{k,+} = M_\mathbb R\setminus s_k^\perp\subset M_\mathbb R$, we have the following transformation of cone decompositions
\[
\mathfrak S^\bullet _{\mu_k^-(\mathbf{s,w})} = T^-_k(\mathfrak S^\bullet_{\mathbf{s,w}})
\]
and
\[
\phi^\bullet_{\mu_k^-(\mathbf{s,w})} = ((T_k^{-})^{\vee})^{-1}\circ \phi^\bullet_{\mathbf{s,w}}\circ (T_k^{-})^{-1}
\]
where $\bullet$ refers to either Hall algebra, quantum stability or stability scattering diagram.
Let $m$ be a generic point on the hyperplane $s_k^\perp$, then we have 
\[
\phi^{q}_{\mu_k^-(\mathbf{s,w})}(m) = \mathbb E(q^{1/2},x^{-s_k})\quad\mathrm{and}\quad\phi^q_{\mathbf{s,w}}(m) = \mathbb E(q^{1/2},x^{s_k});
\]
\[
\phi_{\mu_k^-(\mathbf{s,w})}(m) = \exp(-\mathrm{Li}_2(-x^{-s_k}))\quad\mathrm{and}\quad\phi_{\mathbf{s,w}}(m) = \exp(-\mathrm{Li}_2(-x^{s_k})).
\]
\end{theorem}

Similarly, one can define the following piecewise linear transformation
\[
T_k^+(m)\colon = 
\begin{cases}
 T_k^{-1}(m)\ &\mathrm{if}\ m\in \mathcal H^{k,-}_\mathbf{s},\\
m\ &\mathrm{if}\ m\in \mathcal H^{k,+}_\mathbf{s}.
\end{cases}
\]
Then the above theorem is also true when $-$ is replaced by $+$.

\subsection{The cluster chamber structure} 
In this section, we build the chamber structure for scattering diagrams associated to a non-degenerate SP.
\begin{lemma}\label{positivechamber}
For any SP $(\mathbf{s,w})$, the cones
\[
\mathcal C_\mathbf{s}^{\pm}\colon = \{m\in M_{\mathbb R}|\ m(s_i)> 0\ (< 0)\ ,\ i=1,\dots, n\}
\]
and all their faces are elements in $(\mathfrak S_{\mathbf{s,w}}^{\mathrm{Hall}})^\circ$, $(\mathfrak S_{\mathbf{s,w}}^q)^\circ$ and $\mathfrak S_\mathbf{s,w}^\circ$. We call $\mathcal C_\mathbf{s}^+$ and $\mathcal C_\mathbf{s}^-$ the positive and negative chambers.
\end{lemma}

\begin{proof}
Let $\sigma$ be $\mathcal C_\mathbf{s}^+$ or one of its faces. Then $\sigma$ must be of the form $$\{m\in M_\mathbb R\ |\ m(s_i) = 0\ \mathrm{if}\ i\in P\  \mathrm{and}\ m(s_j)> 0\ \mathrm{if}\ j\in S\}$$ where $S\sqcup P = \{1,\dots, n\}$. It is clear that the $\theta$-semistable subcategory is independent of the choice of $\theta$ in $\sigma$: it is the full subcategory of modules supported on vertices in $P$. These subcategories represent pairwise different elements in the Hall algebra and also stay distinguished after the integration map. So these cones are in fact elements in $(\mathfrak S_{\mathbf{s,w}}^\mathrm{Hall})^\circ$ and also elements in $(\mathfrak S_{\mathbf{s,w}}^q)^\circ$ and $(\mathfrak S_{\mathbf{s,w}})^\circ$. Same argument applies for $\mathcal C_\mathbf{s}^{-}$ and its faces.
\end{proof}

We from now on omit the superscript of $\mathfrak S_{\mathbf{s,w}}^\bullet$ since the rest of the section applies to Hall algebra, quantum stability and stability scattering diagrams.  

\begin{definition}\label{nondegeneracyofpotential}
Let $\mathbf{k} = \{k_1,\dots, k_l\}$ where $k_i\in \{1,\dots, n\}$ be a sequence of indices. We say a QP $(Q,\mathbf{w})$ to be $\mathbf{k}$-\textit{mutable} if for any $
1 \leq l'\leq l$, the QP $\mu_{k_{l'}}\dots\mu_{k_1}(Q,\mathbf{w})$ is $k_{l'}$-mutable. We say $(Q,\mathbf{w})$ to be \textit{non-degenerate} if it is $\mathbf{k}$-mutable for any sequence $\mathbf{k}$. An SP is said to be $\mathbf{k}$-mutable or non-degenerate if the associated QP is.
\end{definition}

Let $(\mathbf{s,w})$ be a non-degenerate SP. By \cref{positivechamber}, the cone $\mathcal C_{\mu_k^-(\mathbf{s})}^+$ is in the cone decomposition $\mathfrak S_{\mu_k^-(\mathbf{s,w})}^\circ$. Note that $\mathcal C_{\mu_k^-(\mathbf{s})}^+$ is contained in the half space $\mathcal H_{\mathbf{s}}^{k,-}$ where the scattering diagrams ${\mathfrak D}_\mathbf{s,w}$ and ${\mathfrak D}_{\mu_k^-(\mathbf{s,w})}$ are equal so the cone $\mathcal C_{\mu_k^-(\mathbf{s})}^+$ is also an element of $\mathfrak S_{\mathbf{s,w}}^\circ$. It is easy to check all faces of $\mathcal C_{\mu_k^-(\mathbf{s})}^+$ are also in $\mathfrak S_{\mathbf{s,w}}^\circ$ by \cref{positivechamber} and \cref{stabilitymutation-}. Thus we have a cone subcomplex of $\mathfrak S_{\mathbf{s,w}}$ consisting of (the closure of) $\mathcal C_{\mu_k^-(\mathbf{s})}^+$ and $\mathcal C_{\mathbf{s}}^+$ and all their faces. In fact $\mathcal C_{\mu_k^-(\mathbf{s})}^+$ and $\mathcal C_{\mathbf{s}}^+$ share a common facet that is contained in $s_k^\perp$, i.e. the cone in $M_\mathbb R$ generated by
\[
\{f_1, \dots, f_n\}\setminus \{f_k\}.
\]
One can continue on performing mutations to expand this cone complex as follows.

Suppose there is a sequence of vertices $\mathbf{k} = \{k_1,\dots, k_l\}$ where $k_i\in \{1,\dots, n\}$. For each $k_i$, we choose a mutation $\mu_{k_i}^{\epsilon_i}$ of SP where $\epsilon_i\in\{+,-\}$, i.e. we choose a sequence of signs $$\epsilon = \{\epsilon_1, \dots, \epsilon_l\}.$$ Since our SP $(\mathbf{s,w})$ is non-degenerate, mutation at each vertex in the sequence is well-defined. Then there is a generated sequence of SPs:
\[
\{(\mathbf{s,w}),\ \mu_{k_1}^{\epsilon_1}(\mathbf{s,w}),\ \dots,\ \mu_\mathbf k^\epsilon(\mathbf{s,w}) \colon = \mu_{k_l}^{\epsilon_l}\cdots\mu_{k_1}^{\epsilon_1}(\mathbf{s,w})\}.
\]
Now we have the positive chamber $\mathcal C^+_{\mu_\mathbf k^\epsilon(\mathbf{s})}\in \mathfrak S_{{\mu_\mathbf k^\epsilon}(\mathbf{s,w})}^\circ$. Consider the piecewise linear map 
\[
T_{\mathbf{k}}^\epsilon = T_{k_l}^{\epsilon_l}\cdots T_{k_1}^{\epsilon_1}\colon M_\mathbb R\rightarrow M_\mathbb R 
\]
which connects two scattering diagrams $\mathfrak D_{\mathbf{s,w}}$ and $\mathfrak D_{\mu_{\mathbf{k,\epsilon}}(\mathbf{s,w})}$. Note that the map $T_{\mathbf{k}}^\epsilon$ preserves top-dimensional cones by \cref{scatmutation}. Pulling back $\mathcal C^+_{\mu_{\mathbf k,\epsilon}(\mathbf{s})}$ by $(T_{\mathbf{k}}^\epsilon)^{-1}$, we have a top-dimensional simplicial cone in the cone decomposition $\mathfrak S_\mathbf{s,w}^\circ$, i.e. we have 
\[
\mathcal C^+_{\mathbf{s,k}}\colon = \mathcal (T_{\mathbf{k}}^\epsilon)^{-1}(\mathcal C^+_{\mu_{\mathbf k,\epsilon}(\mathbf{s})})\in \mathfrak S_\mathbf{s,w}^\circ.
\]
By induction on the length of $\mathbf k$, this cone is independent of the sequence of signs $\epsilon$, which justifies the notation $\mathcal C^+_{\mathbf{s,k}}$. In addition, all the faces of $\mathcal C^+_{\mu_\mathbf k(\mathbf{s})}$ belong to $\mathfrak S_\mathbf {s,w}^\circ$ by \cref{rationalconeandfaces}.

We define an infinite oriented rooted tree $\mathfrak T$ as in \cite[Appendix A]{gross2018canonical}. For each vertex in $\mathfrak T$, there are $n$ outgoing edges labeled by $\{1,\dots, n\}$. We assign the initial SP $(\mathbf{s,w})$ to the root and thus denote the tree by $\mathfrak T_\mathbf{s}$. Given a vertex $v\in \mathfrak T_\mathbf{s}$, take the unique oriented path from the root to $v$. This gives a sequence $\mathbf k(v)$ and by the above discussion there is a well-defined cone 
\[
\mathcal C^+_v\colon = \mathcal C^+_{\mathbf{s,k}(v)}\in \mathfrak S_{\mathbf{s,w}}^\circ.
\]
We can similarly define the negative version $\mathcal C_v^-$ by pulling back $\mathcal C^-_{\mu_{\mathbf k,\epsilon}(\mathbf{s})}$ using the same piecewise linear transformation $\mathcal (T_{\mathbf{k}}^\epsilon)^{-1}$.

\begin{theorem}[The cluster chamber structure]\label{chamberstructure}
For any non-degenerate seed with potential $(\mathbf{s,w})$, the set $\Delta_\mathbf{s,w}^+$ consisting of cones $\overline{\mathcal C^+_{v}}$ where $v$ runs through the set of vertices of $\mathfrak T_\mathbf{s}$ and their faces form a simplicial cone complex. It is a sub poset of $\mathfrak S_{\mathbf{s,w}}^\bullet$ where $\bullet$ represents either Hall algebra, quantum stability or stability scattering diagram. Moreover, for any facet $\sigma$ of any $\overline{\mathcal C_{v}^+}$ with primitive normal vector $n_0\in N_\mathbf{s}^+$, the wall-crossings are given by
\[
\phi_{\mathbf{s,w}}^q(\sigma) = \mathbb E(q^{1/2},x^{n_0}) = \exp\left(\sum_{k\geq 1}\frac{(-1)^{k-1}\hat x^{kn_0}}{k[k]_q}\right)
\]
and
\[
\phi_{\mathbf{s,w}}(\sigma) = \exp(-\mathrm{Li}_2(-x^{n_0})) = \exp\left(\sum_{k\geq 1}\frac{(-1)^{k-1} x^{kn_0}}{k^2}\right).
\]
The same is true for the similarly defined $\Delta_{\mathbf{s,w}}^-$.
\end{theorem}

\begin{proof}
This is simply a consequence of above constructions of $\Delta^\pm_{\mathbf{s,w}}$ and \cref{scatmutation}. Each $\overline{\mathcal C_v^+}$ is a top-dimensional simplicial cone. If two vertices $v$ and $u$ are adjacent in the tree $\mathfrak T_\mathbf{s}$, the cones $\overline{\mathcal C_v^+}$ and $\overline{\mathcal C_u^+}$ intersect at their common facet. Thus we obtain a simplicial cone complex. The wall-crossings at facets are governed by \cref{scatmutation}: it is always of the form $\mathbb E(q^{1/2},x^{n_0})$ in the quantum version or $-\mathrm{Li}_2(-x^{n_0})$ in the semi-classical limit where $n_0$ is the primitive normal vector in $N_\mathbf{s}^+$.
\end{proof}

\begin{remark}\label{coherence}
We see implicitly from the above theorem that for any codimension one cone in $\Delta_\mathbf{s,w}^\pm$, it always has a normal vector in the positive cone $N_\mathbf{s}^+$. The cone complexes $\Delta_\mathbf{s,w}^\pm$ are also independent of the choice of a non-degenerate potential $\mathbf{w}$. Any non-degenerate potential for $\mathbf{s}$ will lead to the same simplicial cone complex. 
\end{remark}

%% file: 5_application.tex
In this section, we explain some applications of the results from previous sections in cluster theory and Donaldson-Thomas theory.

\subsection{In cluster theory}
The cluster scattering diagram $\mathfrak D_\mathbf{s}$ (see \cref{clustersd}) is used in \cite{gross2018canonical} to study the cluster algebra $\mathcal A(\mathbf{s})$ associated to the quiver $Q(\mathbf{s})$. For the definition and related notions of cluster algebras, we refer the readers to \cite{fomin2002cluster} and \cite{gross2018canonical}.

\subsubsection{} Similar to the stability scattering diagrams (\cref{scatmutation}), cluster scattering diagrams also have the following description of mutations.

\begin{theorem}[{\cite[theorem 1.24]{gross2018canonical}}] \label{ghkkmutation}
Let $\mathbf s$ be a seed of $N$. Restricted on the subset $\mathcal H_\mathbf{s}^{k,-}\sqcup \mathcal H_\mathbf{s}^{k,+} = M_\mathbb R\setminus s_k^\perp\subset M_\mathbb R$, we have the following transformation of cone decompositions
\[
\mathfrak S _{\mu_k^-(\mathbf{s})} = T^-_k(\mathfrak S_{\mathbf{s}})
\]
and
\[
\phi_{\mu_k^-(\mathbf{s})} = ((T_k^{-})^{\vee})^{-1}\circ \phi_{\mathbf{s}}\circ (T_k^{-})^{-1}
\]
Let $m$ be a generic point on the hyperplane $s_k^\perp$. Then we have
\[
\phi_{\mu_k^-(\mathbf{s})}(m) = \exp(-\mathrm{Li}_2(-x^{-s_k}))\quad\mathrm{and}\quad\phi_{\mathbf{s}}(m) = \exp(-\mathrm{Li}_2(-x^{s_k})).
\]

\end{theorem}

The following is a direct corollary of \cref{scatmutation} and \cref{ghkkmutation}.
\begin{corollary}\label{mutationinvariance}
Let $(\mathbf{s,w})$ be a $k$-mutable SP. Then the scattering diagrams $\mathfrak D_\mathbf{s,w}$ and $\mathfrak D_\mathbf{s}$ are equal if and only if their mutations $\mathfrak D_{\mu_k^-(\mathbf{s,w})}$ and $\mathfrak D_{\mu_k^-(\mathbf{s})}$ are equal.
\end{corollary}

Using \cref{ghkkmutation}, one can build a simplicial cone complex $\Delta_{\mathbf{s}}^+$ as in \cref{chamberstructure} (see \cite[1.30]{gross2018canonical}). Then $\Delta_{\mathbf{s}}^+$ is exactly the same as $\Delta_{\mathbf{s,w}}^+$ for any non-degenerate $(\mathbf{s,w})$. In particular, as already mentioned in \cref{coherence}, the cone complex $\Delta_{\mathbf{s,w}}^+$ does not depend on the choice of a non-degenerate $\mathbf{w}$, so we denote it simply by $\Delta_\mathbf{s}^+$ and we call it \textit{the (positive) cluster complex} (similarly for the negative cluster complex $\Delta_\mathbf{s}^-$). To summarize, we have the following theorem as a corollary of \cref{scatmutation} and \cref{ghkkmutation}.

\begin{theorem}\label{agreecc}
Let $(\mathbf{s,w})$ be a non-degenerate SP. Then the associated cluster scattering diagram $\mathfrak D_\mathbf{s}$ and the stability scattering diagram $\mathfrak D_\mathbf{s,w}$ have the same wall-crossings in the positive and negative cluster complex $\Delta_\mathbf{s}^\pm$. More precisely, the cone complexes $\Delta^\pm_\mathbf s$ are sub posets simultaneously of $\mathfrak S_\mathbf{s}$ and of $\mathfrak S_\mathbf{s,w}$. The associated functions $\phi_\mathbf{s}$ and $\phi_\mathbf{s,w}$ have the same value on any codimension one cone of $\Delta_\mathbf{s}^\pm$.
\end{theorem}

\begin{definition}
We say that the quiver $Q(\mathbf {s})$ has a green-to-red sequence (see also equivalent definitions in \cite{muller2015existence} and \cite{keller2011cluster}) if (the closure of) the negative cluster chamber $\mathcal C_\mathbf{s}^-$ belongs to $\Delta_\mathbf{s}^+$. It is clear that this is a property independent of the seed but only of the quiver.
\end{definition}

\begin{corollary}\label{reddening}
Let $(\mathbf{s,w})$ be a non-degenerate SP. If $Q(\mathbf{s})$ has a green-to-red sequence, then the scattering diagrams $\mathfrak D_\mathbf{s}$ and $\mathfrak D_\mathbf{s,w}$ are equal, i.e. the cluster scattering diagram of $\mathbf{s}$ and the stability scattering diagram of $(\mathbf{s,w})$ are exactly the same.
\end{corollary}

\begin{proof}
The existence of a green-to-red sequence gives a path on $\mathfrak T_\mathbf{s}$ from the root to a vertex $v$ such that $\mathcal C^+_v = \mathcal C^-_\mathbf{s}$. This path on $\mathfrak T_\mathbf{s}$ gives a path in $M_\mathbb R$ from $\mathcal C^+_\mathbf{s}$ to $\mathcal C^-_\mathbf{s}$ only crossing codimension one cones in $\Delta_\mathbf{s}^+$. These wall-crossings are the same for $\phi_\mathbf{s}$ and $\phi_\mathbf{s,w}$ by \cref{agreecc}. Therefore, $\phi_\mathbf{s}(0)$ and $\phi_\mathbf{s,w}(0)$ as path-ordered products of wall-crossings are equal. That is, they correspond to the same element in $\exp(\hat {\mathfrak g}_\mathbf{s})$ thus defining the same consistent $\mathfrak g_\mathbf{s}$-scattering diagram.
\end{proof}

\begin{remark}
Bridgeland proves in \cite{bridgeland2016scattering} that $\mathfrak D_\mathbf{s} = \mathfrak D_{\mathbf{s},0}$ when $Q(\mathbf{s})$ is acyclic (thus only $\mathbf w = 0$ is possible) with the assumption that the map $p^*$ in \cref{indata} is non-degenerate. The non-degeneracy is not important here since it is not required in the definition of cluster scattering diagrams in \cref{clustersd} whereas it is required in \cite{gross2018canonical} (if not, add principle coefficients). Acyclic quivers are known to have green-to-red sequences. So \cref{reddening} is a generalization of Bridgeland's result to the class of quivers with green-to-red sequences.
\end{remark}

\begin{remark}
In \cite{gross2018canonical}, a distinguished set of elements of the cluster algebra $\mathcal A(\mathbf s)$ indexed by the lattice points in the cluster complex $\Delta_\mathbf{s}^+$ are constructed using the cluster scattering diagram $\mathfrak D_\mathbf{s}$. They are proven to correspond to the cluster monomials of $\mathcal A(\mathbf{s})$. Since the scattering diagrams $\mathfrak D_\mathbf{s,w}$ and $\mathfrak D_\mathbf{s}$ agree on the cluster complex $\Delta_\mathbf{s}^+$, similar construction is valid for the stability scattering diagram $\mathfrak D_\mathbf{s,w}$. Then one recovers the Caldero-Chapoton formula for cluster monomials as in \cite{nagao2013donaldson}.
\end{remark}

\subsubsection{} The quantum version of \cref{ghkkmutation} is still true for $\mathfrak D^q_\mathbf{s}$ using almost the same argument as in \cite{gross2018canonical}; see Appendix A of \cite{davison2019strong} for a detailed proof. In particular, for a generic point $m$ in $s_k^\perp$, we have
\[
\phi^{q}_{\mu_k^-(\mathbf{s})}(m) = \mathbb E(q^{1/2},x^{-s_k})\quad\mathrm{and}\quad\phi^q_{\mathbf{s}}(m) = \mathbb E(q^{1/2},x^{s_k}).
\]

If we assume the non-degenerate SP $(\mathbf{s,w})$ to be polynomial, the quantum stability scattering diagram $\mathfrak D_\mathbf{s,w}^q$ is well-defined. Then \cref{mutationinvariance}, \cref{agreecc} and \cref{reddening} are all valid for $\mathfrak D_\mathbf s^q$ and $\mathfrak D_\mathbf{s,w}^q$.

\subsection{In Donaldson-Thomas theory}\label{dtscattering} There is a slightly different torus Lie algebra
\[
\mathfrak h \colon = \bigoplus_{d\in N}\mathbb Q\cdot x^d,\quad [x^{d_1},x^{d_2}] =  (-1)^{\{d_1,d_2\}}\{d_1,d_2\}x^{d_1+d_2}.
\]
With the multiplication $x^{d_1}\cdot x^{d_2} = (-1)^{\{d_1,d_2\}}x^{d_1+d_2}$, we have a Poisson algebra structure on $\mathfrak h$. We define a twisted product $*$ different from (\ref{twistedprod}) on $T_q$ by setting
\[
x^{d_1}*x^{d_2} = (-q^{\frac{1}{2}})^{\{d_1,d_2\}}x^{d_1+d_2}.
\]
The Poisson algebra $\mathfrak h$ can be considered as the semi-classical limit of associative algebra $(T_q,*)$ at $q^{1/2} = 1$. We also have graded Lie algebras $\mathfrak h_\mathbf{s}$ and $\mathfrak h_{\mathbf s}^q$ as before with a choice of seed $\mathbf s$.

\begin{theorem}[\cite{joyce2012theory}, \cite{bridgeland2012introduction}]\label{dtintegration}
There is an $N^\oplus_\mathbf{s}$-graded Poisson homomorphism 
  \[
  \bar I\colon H_\mathrm{reg}(\mathbf{s,w})\rightarrow  \mathbb Q[N^\oplus_\mathbf{s}],\quad \bar I\left({[X\overset{f}{\rightarrow} \mathfrak M_d]}\right)= \chi(X, f^*(\nu_{\mathfrak M_d}))x^d
  \]
where $\chi(X, f^*(\nu_{\mathfrak M_d}))$ is the weighted Euler characteristic of $X_\mathrm{an}$ by the pull back of the Behrend function on $\mathfrak M_d$.
\end{theorem}

\begin{theorem}[\cite{davison2015donaldson}]\label{qdtintegration}
Let $(\mathbf{s,w})$ be a seed with polynomial potential. There exists an $N_{\mathbf{s}}^\oplus$-graded algebra homomorphism 
 \[
 \bar I_q\colon H(\mathbf{s,w})\rightarrow \left(\mathbb Q((q^\frac{1}{2}))[N^\oplus_\mathbf{s}],\,*\right).
 \]
 It induces an algebra homomorphism of subalgebras
 \[
 \bar I_q\colon H_\mathrm{reg}(\mathbf{s,w})\rightarrow \left(\mathbb Q_\mathrm{reg}(q^\frac{1}{2})[N^\oplus_\mathbf{s}],\,*\right)
 \]
 whose semi-classical limit at $q^{1/2} = 1$ is exactly $\bar I$ in \cref{dtintegration}.
\end{theorem}

\begin{remark}
The algebra homomorphisms in the above two theorems induce two Lie algebra homomorphisms
\[
\bar {\mathcal I} \colon \mathfrak g^{\mathrm{reg}}_{\mathbf{s,w}}\rightarrow \mathfrak h_\mathbf{s}
\quad
\mathrm{and}
\quad
\bar {\mathcal I}_q\colon \mathfrak g_\mathbf{s,w}^\mathrm{reg}\rightarrow \mathfrak h_\mathbf{s}^q.
\]
\end{remark}

Applying the integration maps $\bar {\mathcal I}$ and $\bar {\mathcal I}_q$ to the Hall algebra scattering diagram $\mathfrak D_{\mathbf{s,w}}^\mathrm{Hall}$, we have an $\mathfrak h_\mathbf{s}$-scattering diagram denoted by ${\mathfrak D}_\mathbf{s,w}^\mathrm{DT}$ and an $\mathfrak h^q_\mathbf{s}$-scattering diagram denoted by ${\mathfrak D}_{\mathbf{s,w}}^{q\mathrm{DT}}$. It is reasonable to call them \textit{the Donaldson-Thomas scattering diagram} and \textit{the refined or quantum Donaldson-Thomas scattering diagram} since they encode the DT invariants and the refined DT invariants for $(\mathbf{s,w})$ respectively.

On the other hand, there is a similar construction to the cluster scattering diagram $\mathfrak D_\mathbf{s}$ by using the initial data 
\[
\bar g_{i} = \exp\left(\sum_{k=1}^\infty \frac{x^{ks_i}}{k^2}\right)\in \exp(\hat{\mathfrak h}_{s_i}^{||})
\]
for each $s_i$ as in (\ref{initialdata}). We denoted this $\mathfrak h_\mathbf{s}$-scattering diagram by ${\mathfrak D}_\mathbf{s}^\mathrm{In}$. Similarly we have a quantum version corresponding to the initial data
\[
\bar g_{i,q} = \exp\left(\sum_{k=1}^\infty \frac{x^{k s_i}}{k(q^{k/2}-q^{-k/2})}\right) = \sum_{k=0}^\infty \frac{(-q^{\frac{1}{2}})^{k^2}x^{ks_i}}{[\mathrm{GL}_k]_q}\in \exp(\hat{\mathfrak h}_{s_i}^{q,||})
\]
for each $s_i$. We denote this scattering diagram by ${\mathfrak D}^{q\mathrm{In}}_{\mathbf{s}}$.

\begin{conjecture}[{\cite[conjecture 3.3.4]{kontsevich2014wall}}]\label{ksconjecture} Let $(\mathbf{s,w})$ be a non-degenerate seed with polynomial potential. Then we have
\begin{enumerate}
    \item The scattering diagrams $ {\mathfrak D}_\mathbf{s,w}^\mathrm{DT}$ and $ {\mathfrak D}_\mathbf{s}^\mathrm{In}$ only differ by a central element in $\exp(\hat {\mathfrak h}_\mathbf{s})$.
    \item The scattering diagrams ${\mathfrak D}^{q\mathrm{DT}}_{\mathbf{s,w}}$ and $ {\mathfrak D}^{q\mathrm{In}}_{\mathbf{s}}$ only differ by a central element in $\exp(\hat {\mathfrak h}_\mathbf{s}^q)$.
\end{enumerate}
\end{conjecture}

\begin{theorem}\label{dtreddening}
The above conjecture is true for any non-degenerate seed with (polynomial for the quantum case) potential $(\mathbf{s,w})$ with a green-to-red sequence.
\end{theorem}

\begin{proof}
The proof follows the same routine as the proof of \cref{reddening}. The cluster complex structures of the DT scattering diagrams ${\mathfrak D}^{q\mathrm{DT}}_{\mathbf{s,w}}$ and ${\mathfrak D}^\mathrm{DT}_{\mathbf{s,w}}$ is guaranteed by the cluster complex structure of $\mathfrak D_{\mathbf{s,w}}^\mathrm{Hall}$. What is missing is the cluster complex structures of ${\mathfrak D}_\mathbf{s}^{q\mathrm{In}}$ and ${\mathfrak D}^\mathrm{In}_\mathbf{s}$ (see \cref{ghkkmutation} and \cref{agreecc}), which can be obtained using the same argument as in \cite{gross2018canonical} for cluster scattering diagrams.
\end{proof}

\begin{remark}
Since a scattering diagram $\mathfrak D = (\mathfrak S, \phi)$ is determined by the group element $\phi(0)$ in $\hat G$, $(2)$ of the above conjecture means $ \phi^q_{\mathbf{s,w}}(0) = \phi^q_{\mathbf{s}}(0)\cdot g_0$ where $\log(g_0)$ is in $(\hat {\mathfrak g}_{\mathbf{s}}^q)_{N_0}$ and $N_0\colon = \ker p^*$ (see \ref{indata}). Note that $(2)$ implies $(1)$ by specializing $q^{1/2} =1$ for polynomial potentials. 
\end{remark}

\subsection{Examples}
Instead of the DT scattering diagrams, one can formulate the following conjecture analogous to conjecture \ref{ksconjecture} for stability scattering diagrams.

\begin{conjecture}\label{clusterstabilityconjecture}
Let $(\mathbf{s,w})$ be a non-degenerate seed with potential. Then we have
\begin{enumerate}
     \item The scattering diagrams $ {\mathfrak D}_\mathbf{s,w}$ and $ {\mathfrak D}_\mathbf{s}$ only differ by a central element in $\exp(\hat {\mathfrak g}_\mathbf{s})$.
    \item The scattering diagrams $ {\mathfrak D}^q_{\mathbf{s,w}}$ and $ {\mathfrak D}^q_{\mathbf{s}}$ only differ by a central element in $\exp(\hat {\mathfrak g}_\mathbf{s}^q)$.
\end{enumerate}
\end{conjecture}

In fact, \cref{reddening} confirms this conjecture in the case that the quiver $Q(\mathbf{s})$ has a green-to-red sequence and in this case, the equality holds. The following example (the Markov quiver) is the first evidence of the above conjecture when the two scattering diagrams are not known to be equal. 

\begin{example}\label{markovquiver}
The following quiver is called the Markov quiver. 
\[\begin{tikzcd}[arrow style=tikz,>=stealth,row sep=4em]
1 \arrow[rr, shift right = .4ex] \arrow[rr, shift left = .4ex]
&&
2 \arrow[dl,shift left=.4ex] \arrow[dl,shift right=.4ex]
\\
& 3\arrow[ul,shift left=.4ex] \arrow[ul,shift right=.4ex]
  \end{tikzcd}\]
One can fix a lattice $N = \mathbb Z^3$ with the standard basis $\mathbf s = \{e_1, e_2, e_3\}$. The equipped skewsymmetric form $\{,\}$ is given by the following adjacency matrix 
\[
B = \begin{bmatrix}
0&2&-2\\
-2&0&2\\
2&-2&0\\
\end{bmatrix}.
\]
Then the quiver $Q(\mathbf{s})$ is the Markov quiver. 
\end{example}

\begin{proposition}
Let $(N,\{,\})$ and $\mathbf{s}$ be the data as above. Suppose $(\mathbf{s,w})$ is a non-degenerate SP. Then \cref{clusterstabilityconjecture} is true for $(\mathbf{s,w})$.
\end{proposition}

\begin{proof}
Consider $n_0 = (1,1,1)\in N$. The vector $n_0$ decomposes $M_\mathbb R$ into two open halves ($H^+$ and $H^-$) and a hyperplane $H = n_0^\perp$. It is well-known that the cluster complex $\Delta_\mathbf{s}^+$ is contained in $\overline {H^+}$; see \cite{chavez2013c}. The intersection $\Delta_\mathbf{s}^+\cap H$ is actually the origin. 

Note that each three-dimensional cone in $\Delta_\mathbf{s}^+$ is a simplicial cone. It has three one-dimensional faces (rays) generated by primitive vectors in $M$. These vectors are called the $g$-\textit{vectors}  and all of them lie on the plane 
\[
H_1 \colon = \{m\in M_\mathbb R\mid m(n_0) = 1\}.
\]
The intersection of a three-dimensional cone $\sigma$ of $\Delta_\mathbf{s}^+$ with $H_1$ is a triangle $T_\sigma$ of which the three vertices are the corresponding $g$-vectors. We refer to figure 1 of \cite{fock2016cluster} for a visualization of this picture. For each $g$-vector $v$, there is a ray $v + \mathbb R_+r_v$ on $H_1$ with end point $v$ and direction $r_v$ that is disjoint from $\Delta _\mathbf{s}^+ \cap H_1$. The two vectors $r_v$ and $v$ generate a (relatively open) two-dimensional cone $S_v$ which is in contained in ${H^+}$ and we have $S_v\cap {H} = \mathbb R_+r_v \subset H$.

One can show that $H_1$ is the union of the triangles $T_\sigma$ and the rays $v+\mathbb R_+r_v$ where $\sigma$ runs over all three-dimensional cones in $\Delta_\mathbf{s}^+$ and $v$ runs over the set of $g$-vectors. This gives a decomposition of $H^+$ into (the relative interiors of) cones in $\Delta_\mathbf{s}^+$ and two-dimensional cones of form $S_v$. 

This decomposition of $H^+$ is the same cone decomposition given by $\mathfrak S_\mathbf{s}^\circ$ restricted to $H^+$. In fact, the function $\phi_{\mathbf{s,w}}$ (or $\phi_\mathbf s$ for the cluster scattering diagram) is constant along each $S_v$ and is determined by the face-crossing $\phi_{\mathbf{s,w}}(\mathbb R_+v)$ at the ray $\mathbb R_+v$ as follows. Note that the positive chamber and the negative chamber relative to $\mathbb R_+v$ (see \cref{consistent}) are connected by mutations: they are both cluster chambers. Therefore $\phi_{\mathbf{s,w}}(\mathbb R_+v)$ is determined by the wall-crossings at walls of the cluster complex $\Delta_\mathbf{s}^+$. The cluster and stability scattering diagrams have identical wall-crossings at walls of $\Delta_\mathbf{s}^+$ by \cref{agreecc}. This gives the same wall-crossing at each $S_v$ for these two scattering diagrams.

The same argument applies to the negative cluster complex $\Delta_\mathbf{s}^-$. Then we have for any $m\in M_\mathbb R \setminus  H$, 
\[
\phi_{\mathbf s} (m) = \phi_{\mathbf {s,w}} (m). 
\]
Consider a path $\gamma$ from $\mathcal C_\mathbf s^+$ to $\mathcal C_\mathbf s^-$ that crosses $H$ only once. The only place the path-ordered products can differ for these two scattering diagrams is the wall-crossing at $H$ which is central in $\exp (\hat {\mathfrak g}_\mathbf s)$. This proves conjecture \ref{clusterstabilityconjecture} for the Markov quiver.
\end{proof}

%% file: 0_Scattering.bbl
\begin{thebibliography}{10}

\bibitem{bernstein1973coxeter}
I.~N. Bernstein, I.~M. Gel'fand, and V.~A. Ponomarev.
\newblock Coxeter functors and {G}abriel's theorem.
\newblock {\em Russian mathematical surveys}, 28(2):17--32, 1973.

\bibitem{bridgeland2012introduction}
T.~Bridgeland.
\newblock An introduction to motivic {H}all algebras.
\newblock {\em Advances in Mathematics}, 229(1):102--138, 2012.

\bibitem{bridgeland2016scattering}
T.~Bridgeland.
\newblock Scattering diagrams, {H}all algebras and stability conditions.
\newblock {\em Algebraic Geometry}, 4(5):523--561, 2017.

\bibitem{caldero2006cluster}
P.~Caldero and F.~Chapoton.
\newblock Cluster algebras as {H}all algebras of quiver representations.
\newblock {\em Comment. Math. Helv}, 81:595--616, 2006.

\bibitem{caldero2006triangulated}
P.~Caldero and B.~Keller.
\newblock From triangulated categories to cluster algebras {II}.
\newblock In {\em Annales Scientifiques de l’Ecole Normale Sup{\'e}rieure},
  volume 39(6), pages 983--1009. Elsevier, 2006.

\bibitem{caldero2008triangulated}
P.~Caldero and B.~Keller.
\newblock From triangulated categories to cluster algebras.
\newblock {\em Inventiones Mathematicae}, 172(1):169--211, 2008.

\bibitem{chavez2013c}
A.~N. Ch{\'a}vez.
\newblock On the c-vectors and g-vectors of the markov cluster algebra.
\newblock {\em S{\'e}minaire Lotharingien de Combinatoire}, 69:B69d, 2013.

\bibitem{cheung2019donaldson}
M.-W. Cheung and T.~Mandel.
\newblock Donaldson-{T}homas invariants from tropical disks.
\newblock {\em arXiv preprint arXiv:1902.05393}, 2019.

\bibitem{davison2018positivity}
B.~Davison.
\newblock Positivity for quantum cluster algebras.
\newblock {\em Annals of Mathematics}, pages 157--219, 2018.

\bibitem{davison2019strong}
B.~Davison and T.~Mandel.
\newblock Strong positivity for quantum theta bases of quantum cluster
  algebras.
\newblock 2019.

\bibitem{davison2015donaldson}
B.~Davison and S.~Meinhardt.
\newblock Donaldson--{T}homas theory for categories of homological dimension
  one with potential.
\newblock {\em arXiv preprint arXiv:1512.08898}, 2015.

\bibitem{derksen2008quivers}
H.~Derksen, J.~Weyman, and A.~Zelevinsky.
\newblock Quivers with potentials and their representations {I}: Mutations.
\newblock {\em Selecta Mathematica}, 14(1):59--119, 2008.

\bibitem{derksen2010quivers}
H.~Derksen, J.~Weyman, and A.~Zelevinsky.
\newblock Quivers with potentials and their representations {II}: applications
  to cluster algebras.
\newblock {\em Journal of the American Mathematical Society}, 23(3):749--790,
  2010.

\bibitem{fock2016cluster}
V.~Fock and A.~Goncharov.
\newblock Cluster {P}oisson varieties at infinity.
\newblock {\em Selecta Mathematica}, 22(4):2569--2589, 2016.

\bibitem{fomin2002cluster}
S.~Fomin and A.~Zelevinsky.
\newblock Cluster algebras {I}: foundations.
\newblock {\em Journal of the American Mathematical Society}, 15(2):497--529,
  2002.

\bibitem{gross2018canonical}
M.~Gross, P.~Hacking, S.~Keel, and M.~Kontsevich.
\newblock Canonical bases for cluster algebras.
\newblock {\em Journal of the American Mathematical Society}, 31(2):497--608,
  2018.

\bibitem{gross2010tropical}
M.~Gross, R.~Pandharipande, and B.~Siebert.
\newblock The tropical vertex.
\newblock {\em Duke Mathematical Journal}, 153(2):297--362, 2010.

\bibitem{joyce2007configurations}
D.~Joyce.
\newblock Configurations in abelian categories. {III}. {S}tability conditions
  and identities.
\newblock {\em Advances in Mathematics}, 215(1):153--219, 2007.

\bibitem{joyce2012theory}
D.~Joyce and Y.~Song.
\newblock {\em A theory of generalized Donaldson-{T}homas invariants}.
\newblock American Mathematical Soc., 2012.

\bibitem{keller2008cluster}
B.~Keller.
\newblock Cluster algebras, quiver representations and triangulated categories.
\newblock {\em arXiv:0807.1960}, 2008.

\bibitem{keller2011cluster}
B.~Keller.
\newblock On cluster theory and quantum dilogarithm identities.
\newblock In {\em Representations of Algebras and Related Topics, Editors A.
  Skowronski and K. Yamagata, EMS Series of Congress Reports, European
  Mathematical Society}, pages 85--11, 2011.

\bibitem{kontsevich2008stability}
M.~Kontsevich and Y.~Soibelman.
\newblock Stability structures, motivic {D}onaldson-{T}homas invariants and
  cluster transformations.
\newblock {\em arXiv preprint arXiv:0811.2435}, 2008.

\bibitem{kontsevich2014wall}
M.~Kontsevich and Y.~Soibelman.
\newblock Wall-crossing structures in {D}onaldson--{T}homas invariants,
  integrable systems and mirror symmetry.
\newblock In {\em Homological mirror symmetry and tropical geometry}, pages
  197--308. Springer, 2014.

\bibitem{muller2015existence}
G.~Muller.
\newblock The existence of a maximal green sequence is not invariant under
  quiver mutation.
\newblock {\em arXiv preprint arXiv:1503.04675}, 2015.

\bibitem{nagao2013donaldson}
K.~Nagao.
\newblock Donaldson--{T}homas theory and cluster algebras.
\newblock {\em Duke Mathematical Journal}, 162(7):1313--1367, 2013.

\bibitem{palu2008cluster}
Y.~Palu.
\newblock Cluster characters for 2-{C}alabi--{Y}au triangulated categories.
\newblock In {\em Annales de l'institut Fourier}, volume 58(6), pages
  2221--2248, 2008.

\bibitem{plamondon2011cluster}
P.-G. Plamondon.
\newblock Cluster characters for cluster categories with infinite-dimensional
  morphism spaces.
\newblock {\em Advances in Mathematics}, 227(1):1--39, 2011.

\bibitem{qin2019bases}
F.~Qin.
\newblock Bases for upper cluster algebras and tropical points.
\newblock {\em arXiv preprint arXiv:1902.09507}, 2019.

\end{thebibliography}
